\documentclass[a4paper,11pt]{amsart}
\pdfoutput=1
\usepackage{amstext,amsfonts,amsthm,graphicx,amssymb,amscd,epsfig}
\usepackage{amsmath}
\usepackage[ansinew]{inputenc}    
\usepackage{psfrag}
\usepackage{mathrsfs}
\usepackage{graphicx}
\usepackage[all,knot,arc,import,poly]{xy}
\usepackage{epsfig}
\usepackage{subfigure}
\usepackage{subfig}
\usepackage{float}
\usepackage{xcolor}
\usepackage{caption}
\usepackage{fancyhdr}

\theoremstyle{definition}
\newtheorem{definition}{Definition}[section]
\newtheorem{ex}[definition]{Example}
\newtheorem{rem}[definition]{Remark}

\theoremstyle{plain}
\newtheorem{prop}[definition]{Proposition}
\newtheorem{lem}[definition]{Lemma}
\newtheorem{coro}[definition]{Corollary}
\newtheorem{teo}[definition]{Theorem}

\newfont{\bbb}{msbm10 scaled\magstephalf}     

\def\reg{\operatorname{reg}}
\def\sing{\operatorname{sing}}
\def\Hess{\operatorname{Hess}}

\title[Relating geometry with projections and normal sections]{Relating second order geometry of manifolds through projections and normal sections}

\author{P. Benedini Riul, R. Oset Sinha}

\date{}

\address{Departamento de Matem\'{a}tica, Universidade Federal de S\~{a}o Carlos, Caixa Postal 676, S\~{a}o Carlos, SP 13560-905, Brazil}

\email{pedro.benedini.riul@gmail.com}

\address{Departament de Matem\`{a}tiques,
Universitat de Val\`encia, Campus de Burjassot, 46100 Burjassot,
Spain}

\email{raul.oset@uv.es}

\thanks{Work of P. Benedini Riul supported by FAPESP Grant 2019/00194-6}
\thanks{Work of R. Oset Sinha partially supported by MICINN Grant PGC2018-094889-B-I00}


\begin{document}

\begin{abstract}
We use normal sections to relate the curvature locus of regular (resp. singular corank 1) 3-manifolds in $\mathbb{R}^6$ (resp. $\mathbb R^5$) with regular (resp. singular corank 1) surfaces in $\mathbb R^5$ (resp. $\mathbb R^4$). For example we show how to generate a Roman surface by a family of ellipses different to Steiner's way. Furthermore, we give necessary conditions for the 2-jet of the parametrisation of a singular 3-manifold to be in a certain orbit in terms of the topological types of the curvature loci of the singular surfaces obtained as normal sections.  We also study the relations between the regular and singular cases through projections. We show there is a commutative diagram of projections and normal sections which relates the curvature loci of the different types of manifolds, and therefore, that the second order geometry of all of them is related. In particular we define asymptotic directions for singular corank 1 3-manifolds in $\mathbb R^5$ and relate them to asymptotic directions of regular 3-manifolds in $\mathbb R^6$ and singular corank 1 surfaces in $\mathbb R^4$.
\end{abstract}

\maketitle

\section{Introduction}

The study of second order geometry of manifolds in Euclidean spaces dates as far back as Gauss. By second order geometry we refer to any geometrical aspects which can be captured by the second fundamental form, or, in modern terminology, by the 2-jet of a parametrisation of the manifold. Concepts such as elliptic/parabolic/hyperbolic points, normal curvature, asymptotic directions and some aspects of the contacts with hyperplanes and spheres are included in the study of second order geometry.

In his seminal paper \cite{Little}, Little studied second order geometry of immersed manifolds in Euclidean spaces of dimensions greater than 3, in particular special attention was given to immersed surfaces in $\mathbb{R}^{4}$. He defined the second fundamental form and the curvature locus, which is an ellipse in this case. The curvature locus is the image in the normal space by the second fundamental form of the unitary tangent vectors. It can also be seen as the curvature vectors of normal hyperplane sections of the surface. The curvature locus is not an affine invariant but its topological type and its position with respect to the origin is an affine invariant. Besides, all the second order geometry is captured by this object.

The introduction of Singularity Theory techniques to study the differential geometry of manifolds in Euclidean spaces has given a great impulse to this subject in the last 20 years. There are many papers devoted to regular surfaces in $\mathbb R^4$ such as  \cite{BruceNogueira,BruceTari,GarciaMochidaFusterRuas,MochidaFusterRuas,MochidaFusterRuas2,BallesterosTari,OsetSinhaTari,RomeroFuster}, amongst others. For surfaces in $\mathbb R^5$ \cite{moraes09,Moraes02,Fuster/Ruas/Tari} are good examples. In fact, there is a recent book which covers these topics (\cite{Livro}). The study of regular 3-manifolds in $\mathbb R^6$ is also very recent. Here the curvature locus is a Veronese surface with many different topological types (see \cite{Carmen3var,Carmen3var2}).

The interest however, both for singularists and differential geometers has turned to the study of singular manifolds (\cite{annals}). For singular corank $1$ surfaces in $\mathbb{R}^{n}$, $n=3,4$ we can cite \cite{BenediniOset,BenediniOset2,Benedini/Sinha/Ruas,MartinsBallesteros}, and for singular corank $1$ 3-manifolds in $\mathbb R^5$, \cite{BenediniRuasSacramento}. Here the curvature locus is a parabola or a parabolic version of a Veronese surfaces. Curvature loci in general have been studied in \cite{nunoromerobringas}, for example.

The aim of this paper is to relate the geometry of all these objects which have traditionally been studied separately. There is a natural relation between regular and singular objects. When projecting a regular $n$-manifold in $\mathbb R^k$ along a tangent direction you obtain a singular $n$-manifold in $\mathbb R^{k-1}$. On the other hand, taking normal hyperplane sections of the $n$-manifold gives a family of $(n-1)$-manifolds in one dimension less. In Section \ref{diagram} we establish a commutative diagram using projections and normal sections which induces a commutative diagram amongst the curvature loci with immersions and blow-ups. As a result of this we prove that the second order geometry of all these objects is related. This justifies known relations for projections when $n=2$ and $k=4$, for example, and motivates to look for further relations between the geometries of different manifolds, both regular and singular, in different Euclidean spaces.

Section \ref{prelim} is devoted to preliminary results on the geometry of all the different objects appearing throughout the paper. In Section \ref{sections} we study normal sections of 3-manifolds both for the regular and singular cases and show that the curvature locus of a 3-manifold can be generated by the curvature loci of the surfaces obtained by normal sections. In particular we show how a Roman Steiner surface or a cross-cap surface can be generated by ellipses. Using these sections we can recover some geometry of the 3-manifold by the topological types of the curvature loci of the sections.

In Section \ref{asymptotic}, inspired by the commutative diagram of Section \ref{diagram}, we define asymptotic directions for singular 3-manifolds in $\mathbb R^5$ and relate them to asymptotic directions of regular 3-manifolds in $\mathbb R^6$ and singular surfaces in $\mathbb R^4$. We prove that the direction of projection is asymptotic if and only if the singularity of the singular projection is not in the best $\mathcal A^2$-orbit. We then explain how this direction of projection becomes a null tangent direction in the singular projection and so justify the existence of infinite asymptotic directions in the singular case, which was not fully understood until now.

Aknowledgements: the authors would like to thank M. A. S. Ruas for useful conversations and constant encouragement.


\section{The geometry of regular and singular surfaces and 3-manifolds in Euclidean spaces}\label{prelim}

\subsection{Regular surfaces in Euclidean spaces}

Given a smooth surface $M^{2}_{\reg}\subset\mathbb{R}^{2+k}$, $k\geq2$ and $f:U\rightarrow\mathbb{R}^{2+k}$
a local parametrisation of $S$ with $U\subset\mathbb{R}^{2}$ an open subset, let
$\{e_{1},\ldots,e_{2+k}\}$ be an orthonormal frame of $\mathbb{R}^{2+k}$ such that at any $u\in U$,
$\{e_{1}(u),e_{2}(u)\}$ is a basis for $T_{p}M^{2}_{\reg}$ and $\{e_{3}(u),\ldots,e_{2+k}(u)\}$ is a basis for
$N_{p}M^{2}_{\reg}$ at $p=f(u)$.

The second fundamental form of $M^{2}_{\reg}$ at a point $p$ is defined by
$II_{p}:T_{p}M^{2}_{\reg}\times T_pM^{2}_{\reg}\rightarrow N_{p}M^{2}_{\reg}$ given by
$$
\begin{array}{cc}
     II_p(e_1(u),e_1(u))=\pi_2(f_{xx}(u)),\  II_p(e_1(u),e_2(u))=\pi_2(f_{xy}(u)) \\
     II_p(e_2(u),e_2(u))=\pi_2(f_{yy}(u)),
\end{array}
$$
in the basis $\{e_{1}(u),e_{2}(u)\}$ of $T_pM^{2}_{\reg}$, where $\pi_2:T_p\mathbb{R}^{2+k}\rightarrow N_pM^{2}_{\reg}$
is the canonical projection on the normal space.
We extend $II_p$ to the hole space in a unique way as a symmetric bilinear map.

Taking $w=w_{1}e_{1}+w_{2}e_{2}\in T_pM^{2}_{\reg}$, we can write the quadratic form
$$II_{p}(w,w)=\sum_{i=1}^{k}(l_{i}w_{1}^{2}+2m_{i}w_{1}w_{2}+n_{i}w_{2}^{2})e_{2+i},$$
where $l_{i}=\langle f_{xx},e_{2+i}\rangle,\ m_{i}=\langle f_{xy},e_{2+i}\rangle$
and $n_{i}=\langle f_{yy},e_{2+i}\rangle$, for $i=1,\ldots,k$, are
called the coefficients of the second fundamental form with respect to the
frame above. The matrix of the second fundamental
form with respect to the orthonormal frame above is given by
$$
\left(
         \begin{array}{ccc}
           l_{1} & m_{1} & n_{1} \\
           \vdots & \vdots & \vdots \\
           l_k & m_k & n_k
         \end{array}
       \right).
$$

Consider a point $p\in M^{2}_{\reg}$ and the unit circle $\mathbb{S}^{1}$ in $T_{p}M^{2}_{\reg}$ parametrised by
$\theta\in [0,2\pi]$. The curvature vectors $\eta(\theta)$ of the normal sections of
$M^{2}_{\reg}$ by the hyperplane $\langle\theta\rangle\oplus N_{p}M^{2}_{\reg}$ form an ellipse in the
normal space $N_{p}M^{2}_{\reg}$, called the \emph{curvature ellipse} of $M^{2}_{\reg}$ at $p$, denoted by $\Delta_e$, which is the same as the image of the map
$\eta:\mathbb{S}^{1}\subset T_{p}M^{2}_{\reg}\rightarrow N_{p}M^{2}_{\reg}$, where
\begin{equation}
\eta(\theta)=\sum_{i=1}^{k}(l_{i}\cos(\theta)^2+2m_{i}\cos(\theta)\sin(\theta)+n_{i}\sin(\theta)^2)e_{2+i}.
\end{equation}\label{ellipse}
Moreover, if we write $u=\cos(\theta) e_{1}+\sin(\theta) e_{2}\in \mathbb{S}^{1}$,
$II_{p}(u,u)=\eta(\theta)$.

\subsection{Second order geometry of $3$-manifolds in $\mathbb{R}^n$}

Let $M^{3}_{\reg}$ be a $3$-manifold in $\mathbb{R}^{3+k}$, $k\geq1$, given locally as the image of the map
$f:U\subset\mathbb{R}^3\rightarrow\mathbb{R}^{3+k}$. Taking $p=f(u)$, the basis of $T_pM^{3}_{\reg}$ is $B^{f}=\{f_x,f_y,f_z\}$,
where $f_x=\frac{\partial f}{\partial x}$, etc. The orthonormal frame $\{e_1,\ldots,e_k\}$ is a frame of $N_pM^{3}_{\reg}$ if the
orientation of the frame $\{f_x.f_y,f_z,e_1,\ldots,e_k\}$ coincides with the orientation of $\mathbb{R}^{3+k}$.

The second fundamental form $II_p:T_pM^{3}_{\reg}\times T_pM^{3}_{\reg}\rightarrow N_pM^{3}_{\reg}$ is the bilinear map given by
$II_p(v,w)=\pi_2(d^2f(v,w))$, that projects the second derivative of $f$ onto the normal space at $p$.
The second fundamental form of $M^{3}_{\reg}$ at $p$ along a normal vector field $\nu$
 is the bilinear map $II_{p}^{\nu}:T_pM^{3}_{\reg}\times T_pM^{3}_{\reg}\rightarrow\mathbb{R}$ defined by
 $II_{p}^{\nu}(v,w)=\langle \nu,d^2f(v,w)\rangle$.

Given a point $p\in M^{3}_{\reg}$, $k\geq1$, and a unit tangent vector $v\in\mathbb{S}^2\subset T_pM^{3}_{\reg}$, the normal 
(codimension 2) section of $M^{3}_{\reg}$ in the direction $v$ is
$\gamma_v= M^{3}_{\reg}\cap\mathcal{H}_v$, where $\mathcal{H}_v=\{\lambda v\}\oplus N_pM^{3}_{\reg}$ is an affine subspace of codimension
$2$ through $p$ in $\mathbb{R}^{3+k}$. We denote by $\eta(v)$ the normal curvature vector of $\gamma_v$ in $N_pM^{3}_{\reg}$. When we vary $v\in\mathbb{S}^2\subset T_pM^{3}_{\reg}$, the vectors $\eta(v)$ describe a surface in the normal space, called the curvature locus of $M^{3}_{\reg}$
at $p$ and denoted by $\Delta_v$

The curvature locus can also be seen as the image of the unit tangent vectors at $p$ via $II(v,v)$.

\begin{lem}\cite{Carmen3var}\label{lemma.curv.media}
Taking spherical coordinates in $\mathbb{S}^2\subset T_pM^{3}_{\reg}$, we can parametrise the curvature locus of $M^{3}_{\reg}$
at $p$ by
$\eta:\mathbb{S}^{2}\subset T_pM^{3}_{\reg}\rightarrow N_pM^{3}_{\reg},\ (\theta,\phi)\mapsto\eta(\theta,\phi)$,
where
$$
\begin{array}{cl}
    \eta(\theta,\phi)= & H+(1+3\cos(2\phi))B_1+\cos(2\theta)\sin(\theta)^2 B_2 \\
     & +\sin(2\theta)\sin(\phi)^2 B_3+cos(\theta)\sin(2\phi)B_4+\sin(\theta)\sin(2\phi)B_5
\end{array}
$$
with
$$
\begin{array}{c}
     H=\frac{1}{3}(f_{xx}+f_{yy}+f_{zz}),\ B_1=\frac{1}{12}(-f_{xx}-f_{yy}+2f_{zz}),  \\
     B_2=\frac{1}{2}(f_{xx}-f_{yy}),\ B_3=f_{xy},\ B_4=f_{xz}\ \mbox{e}\ B_5=f_{yz}.
\end{array}
$$
\end{lem}

The mean curvature of $M^{3}_{\reg}$ at $p$ is given by $H(p)=\frac{1}{3}(f_{xx}+f_{yy}+f_{zz})(p)$, the first normal space
is $N_p^1M^{3}_{\reg}=\langle H,B_1,B_2,B_3,B_4,B_5\rangle_{(p)}$. The affine hull of the curvature locus is denoted by $Aff_p$ and
the linear subspace of $N_p^1M^{3}_{\reg}$ parallel to $Aff_p$ by $E_p$. Also, a point $p\in M^{3}_{\reg}$ is said to be
of type $(M^{3}_{\reg})_i$, $i=0,1,\ldots,6$ if $\dim N_p^1M^{3}_{\reg}=i$.

The curvature locus of a regular $3$-manifold in $\mathbb{R}^{3+k}$ can be seen as the image of the
classical Veronese surface of order $2$ via a convenient linear map.
The expression of this surface is given by the image of the unit sphere in $\mathbb{R}^3$ via the map
$\xi:\mathbb{R}^3\rightarrow\mathbb{R}^6$, with
$\xi(u,v,w)=(u^2,v^2,w^2,\sqrt{2}uv,\sqrt{2}uw,\sqrt{2}vw)$. For more details, see \cite{Carmen3var,Carmen3var2}.

The next theorem describes the possible topological types of the curvature locus of a regular $3$-manifold in $\mathbb{R}^6$.

\begin{teo}\cite{Carmen3var}
The curvature locus at a point $p$ of type $(M^{3}_{\reg})_3$ in a $3$-manifold $M^{3}_{\reg}\subset\mathbb{R}^6$ is isomorphic to one of the following:
\begin{itemize}
    \item[(i)] If $H(p)\in E_p$: A Roman Steiner surface (Figure \ref{fig:steiner1}), a Cross-Cap surface (Figure \ref{fig:cross-cap}), a Steiner surface of type $5$ (Figure \ref{fig:tipo5}), a Cross-Cup surface (Figure \ref{fig:cross-cup}), an ellipsoid or a (compact) cone.
    \item[(ii)] If $H(p)\notin E_p$: An elliptic region, a triangle, a compact planar cone or a planar projection of type $1$, $2$ or $3$ of the Veronese surface.
\end{itemize}
Here, $H(p)$ is the mean curvature vector.

\begin{center}
\begin{figure}
\begin{minipage}[b]{0.45\linewidth} \hspace{1.4cm}
\includegraphics[scale=0.35]{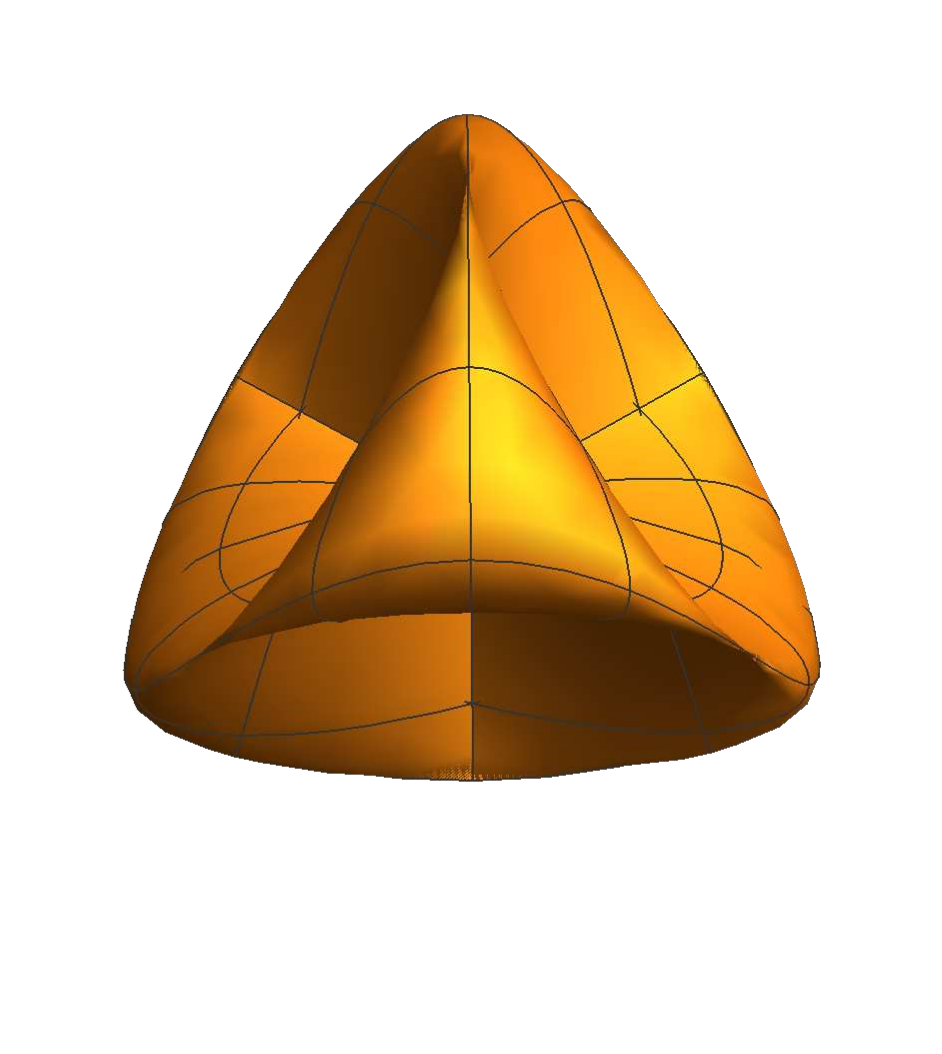}
\caption{Roman Steiner.}
\label{fig:steiner1}
\end{minipage} \hfill
\begin{minipage}[b]{0.45\linewidth} \hspace{1.4cm}
\includegraphics[scale=0.35]{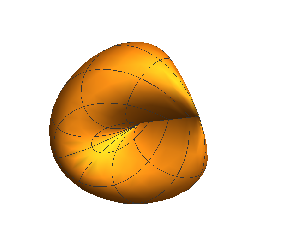}
\caption{Cross-cap.}
\label{fig:cross-cap}
\end{minipage} \hfill \\
\begin{minipage}[b]{0.45\linewidth} \hspace{1cm}
\includegraphics[scale=0.35]{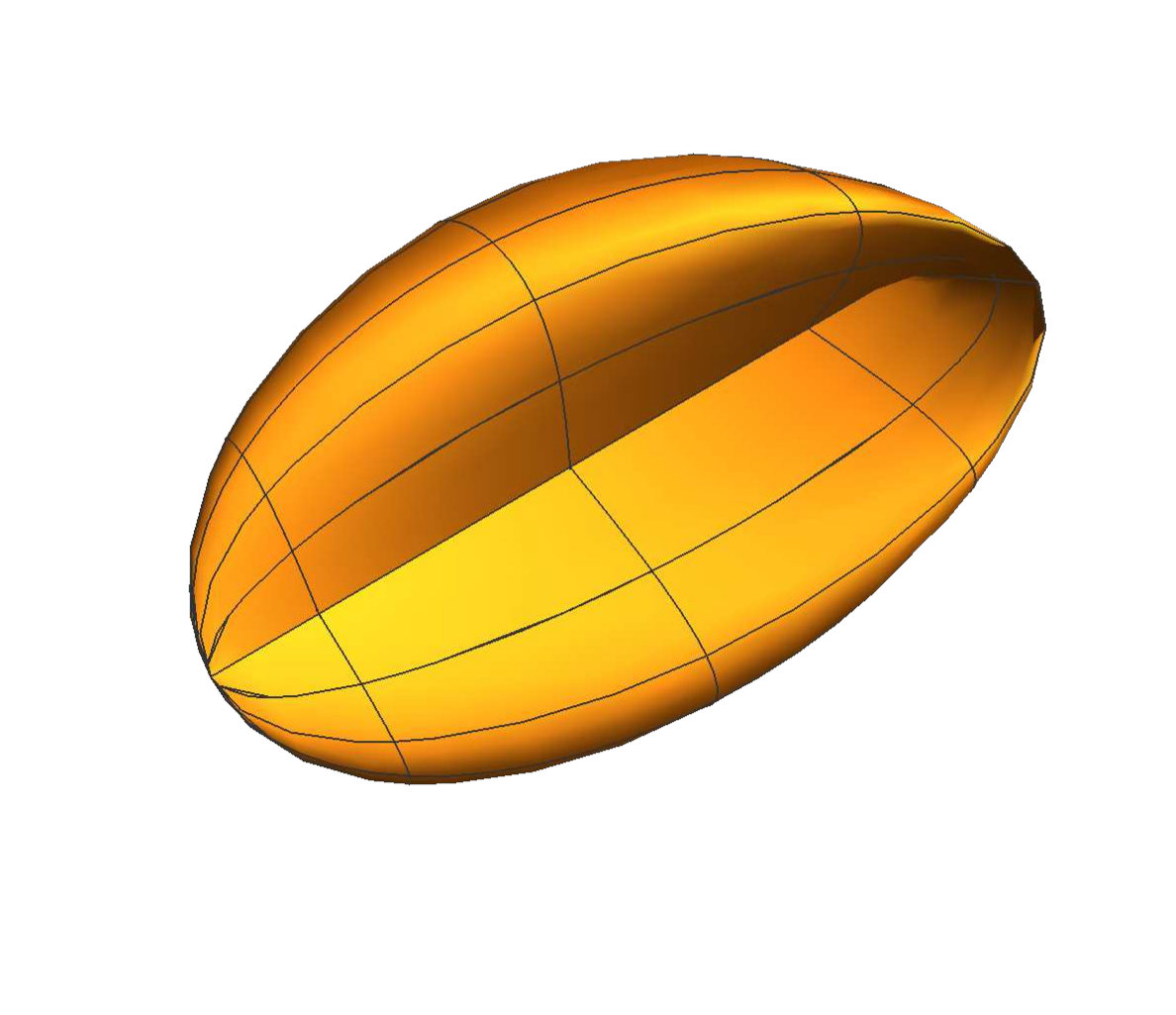}
\caption{Steiner Type $5$.}
\label{fig:tipo5}
\end{minipage} \hfill
\begin{minipage}[b]{0.45\linewidth} \hspace{1.5cm}
\includegraphics[scale=0.35]{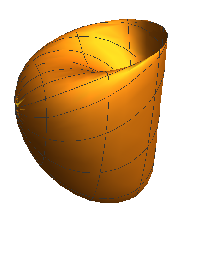}
\caption{Cross-cup.}
\label{fig:cross-cup}
\end{minipage} \hfill
\end{figure}
\end{center}

\end{teo}

The following refers to regular manifolds $M^{n}_{\reg}$ immersed in $\mathbb{R}^{2n}$. For such objects, the definitions of tangent and normal spaces and second fundamental forms are analogues to the case of $3$-manifolds in $\mathbb{R}^{3+k}$ (see \cite{dreibelbis} for details). The curvature Veronese submanifold of $M^{n}_{\reg}$ at $p$ is the image of the unit tangent sphere $\mathbb{S}^{n-1}\subset T_pM^n_{\reg}$ via the second fundamental form.

Let $\{e_1,\ldots,e_n\}$ and $\{v_1,\ldots,v_n\}$ be basis for $T_pM^n_{\reg}$ and $N_pM^n_{\reg}$, respectively. For each $u\in T_pM^n_{\reg}$,
$A(u)$ denotes the $n\times n$ matrix such that $A(u)_{ij}=II_{v_{i}}(e_j,u)=\langle II(e_j,u),v_{i}\rangle$.

\begin{teo}\cite{dreibelbis}\label{Teo-Dreibelbis}
Let $f:M^n_{\reg}\rightarrow\mathbb{R}^{2n}$ be an immersion and let $u$ be a unit tangent vector at $p\in M^n_{\reg}$. The following are
equivalent:
\begin{itemize}
    \item[(i)] The vector $u$ satisfies $\det(A(u))\neq0$;
    \item[(ii)] There exists a unit normal vector $v$ such that $II_v(u,\cdot)=0$;
    \item[(iii)] There exists a height function $h_v(x)=\langle v,x\rangle$ such that $h\circ f$ has a degenerate (non Morse) singularity with $\langle h,d^2fu\rangle=0$;
    \item[(iv)] The vector $II(u,u)$ is tangent to the curvature Veronese submanifold at $II(u,u)$, or the curvature Veronese submanifold has a singularity at $u$.
\end{itemize}
\end{teo}

A unit tangent vector $u\in\mathbb{S}^{n-1}\subset T_pM^n_{\reg}$ is an asymptotic direction  if it satisfies any of the items in Theorem \ref{Teo-Dreibelbis}.

\subsection{Corank $1$ surfaces in $\mathbb{R}^n$, $n=3,4$}


Let $M^{2}_{\sing}$ be a corank $1$ surface in $\mathbb{R}^{n}$, $n=3,4$, at $p$. The singular surface $M^{2}_{\sing}$ will be taken as the image of a smooth map $g:\tilde{M}\rightarrow \mathbb{R}^{n}$, where $\tilde{M}$ is a smooth regular surface and $q\in\tilde{M}$ is a corank $1$ point of $g$ such that $g(q)=p$. Also, consider $\phi:U\rightarrow\mathbb{R}^{2}$ a local coordinate system defined in an open neighbourhood $U$ of $q$ at $\tilde{M}$. Using this construction, we may consider a local parametrisation $f=g\circ\phi^{-1}$ of $M$ at $p$ (see the diagram below).

$$
\xymatrix{
\mathbb{R}^{2}\ar@/_0.7cm/[rr]^-{f} & U\subset\tilde{M}\ar[r]^-{g}\ar[l]_-{\phi} & M^{2}_{\sing}\subset\mathbb{R}^{n}
}
$$

The tangent line of $M^{2}_{\sing}$ at $p$, $T_{p}M^{2}_{\sing}$, is given by $\mbox{Im}\ dg_{q}$, where $dg_{q}:T_{q}\tilde{M}\rightarrow T_{p}\mathbb{R}^{4}$ is the differential map of $g$ at $q$. Thus, the normal space of $M^{2}_{\sing}$ at $p$, $N_{p}M^{2}_{\sing}$, is the subspace satisfying $T_{p}M^{2}_{\sing}\oplus N_{p}M^{2}_{\sing}=T_{p}\mathbb{R}^{n}$.

Consider the orthogonal projection $\pi_2:T_{p}\mathbb{R}^{n}\rightarrow N_{p}M^{2}_{\sing}$, $w\mapsto \pi_2(w)$.
The first fundamental form of $M^{2}_{\sing}$ at $p$, $I:T_{q}\tilde{M}\times T_{q}\tilde{M}\rightarrow \mathbb{R}$ is given by
$$I(u,v)=\langle dg_{q}(u),dg_{q}(v)\rangle,\ \ \forall\ u,v\in T_{q}\tilde{M}.$$

The second fundamental form of $M^{2}_{\sing}$ at $p$, $II:T_{q}\tilde{M}\times T_{q}\tilde{M}\rightarrow N_{p}M$ in the basis $\{\partial_{x},\partial_{y}\}$ of $T_{q}\tilde{M}$ is given by
$$
\begin{array}{c}
     II(\partial_{x},\partial_{x})=\pi_2(f_{xx}(\phi(q))),\  II(\partial_{x},\partial_{y})=\pi_2(f_{xy}(\phi(q))),\
      II(\partial_{y},\partial_{y})=\pi_2(f_{yy}(\phi(q)))
\end{array}
$$
and we extend it to the whole space in a unique way as a symmetric bilinear map.

Given a normal vector $\nu\in N_{p}M^{2}_{\sing}$, we define the second fundamental form along $\nu$, $II_{\nu}:T_{q}\tilde{M}\times T_{q}\tilde{M}\rightarrow\mathbb{R}$ given by $II_{\nu}(u,v)=\langle II(u,v),\nu\rangle$, for all $u,v\in T_{q}\tilde{M}$.
The coefficients of $II_{\nu}$ with respect to the basis $\{\partial_{x},\partial_{y}\}$ of $T_{q}\tilde{M}$ are
$$
\begin{array}{cc}
     l_{\nu}(q)=\langle \pi_2(f_{xx}),\nu\rangle(\phi(q)),\ m_{\nu}(q)=\langle \pi_2(f_{xy}),\nu\rangle(\phi(q)),  \\
     n_{\nu}(q)=\langle \pi_2(f_{yy}),\nu\rangle(\phi(q)).
\end{array}
$$
Given $u=\alpha\partial_{x}+\beta\partial_{y}\in T_{q}\tilde{M}$ and
fixing an orthonormal frame $\{\nu_{1},\ldots,\nu_{n-1}\}$ of $N_{p}M^{2}_{\sing}$,
$$
\begin{array}{cl}\label{eq.2ff}
II(u,u) & =II_{\nu_{1}}(u,u)\nu_{1}+\ldots+II_{\nu_{n-1}}(u,u)\nu_{n-1} \\
        & =\displaystyle{\sum_{i=1}^{n-1}(\alpha^{2}l_{\nu_{i}}(q)+2\alpha\beta m_{\nu_{i}}(q)+\beta^{2}n_{\nu_{i}}(q))\nu_{i}},
\end{array}
$$
$n=3,4$.
Moreover, the second fundamental form is represented by the matrix of coefficients
$$
\left(
  \begin{array}{ccc}
    l_{\nu_{1}} & m_{\nu_{1}} & n_{\nu_{1}} \\
    \vdots & \vdots & \vdots \\
    l_{\nu_{n-1}} & m_{\nu_{n-1}} & n_{\nu_{n-1}} \\
  \end{array}
\right).
$$

Let $C_{q}\subset T_{q}\tilde{M}$ be the subset of unit tangent vectors and let $\eta:C_{q}\rightarrow N_{p}M$ be the map given by $\eta(u)=II(u,u)$. The \emph{curvature parabola} of $M^{2}_{\sing}$ at $p$, denoted by $\Delta_{p}$, is the subset $\eta(C_q)$.

The curvature parabola is a plane curve for both cases, $n=3$ and $n=4$, and it can degenerate into a half-line, a line or even a point.
This special curve plays a similar role as the curvature ellipse does for regular surfaces. Therefore, it contains information about the second order geometry of the surface.

For the case $n=3$, there is no doubt about which plane contains $\Delta_p$, since the normal space is a plane. However, if $n=4$, the normal space has dimension $3$ and if $\Delta_p$ degenerates, we need to be careful.
Let $M^{2}_{\sing}\subset\mathbb{R}^4$ be a corank $1$ surface at $p$. The minimal affine space which contains the curvature parabola is denoted by $Aff_{p}$. The plane denoted by $E_{p}$ is the vector space: parallel to $Aff_{p}$ when $\Delta_{p}$ is a non degenerate parabola, the plane through $p$ that contains $Aff_{p}$ when $\Delta_{p}$ is a non radial half-line or a non radial line and any plane through $p$ that contains $Aff_{p}$ when $\Delta_{p}$ is a radial half-line, a radial line or a point.

A non zero direction $u\in T_{q}\tilde{M}$ is called asymptotic if there is a non zero vector $\nu\in N_{p}M^{2}_{\sing}$ (for $n=3$) or $\nu\in E_{p}$ (for $n=4$) such that
$II_{\nu}(u,v)=\langle II(u,v),\nu\rangle=0$ for all $v\in T_{q}\tilde{M}$.
Moreover, in such case, we say that $\nu$ is a binormal direction.

For the case $n=4$, the normal directions $\nu\in N_{p}M^{2}_{\sing}$ which are not in the plane $E_p$ but also satisfy the condition $II_{\nu}(u,v)=\langle II(u,v),\nu\rangle=0$, are called degenerate directions. The subset of degenerate directions in $N_pM^2_{\sing}$
is a cone and the binormal direcitons are those in the intersection of this cone with $E_p$.

It is possible to take a coordinate system $\phi$ and make rotations in the target in order to obtain
$f(x,y)=g\circ\phi^{-1}(x,y)=(x,f_{2}(x,y),\ldots,f_{n}(x,y))$,
where $\frac{\partial f_{i}}{\partial x}(\phi(q))=\frac{\partial f_{i}}{\partial y}(\phi(q))=0$ for $i=2,\ldots,n$. Taking an orthonormal frame $\{\nu_{1},\ldots,\nu_{n-1}\}$ of $N_{p}M^2_{\sing}$, the curvature parabola $\Delta_{p}$ can be parametrised by
$$\eta(y)=\sum_{i=1}^{n-1}(l_{\nu_{i}}+2m_{\nu_{i}}y+n_{\nu_{i}}y^{2})\nu_{i}.$$
Each parameter $y\in\mathbb{R}$ corresponds to a unit tangent direction $u=\partial_{x}+y\partial_{y}=(1,y)\in C_{q}$. We denote by $y_{\infty}$ the parameter corresponding to the null tangent direction given by $u=\partial_{y}=(0,1)$. For each possibility of $\Delta_{p}$ we define $\eta(y_{\infty})$: when $\Delta_{p}$ is a line or a half-line $\eta(y_{\infty})=\eta'(y_{\infty})=\eta'(y)/|\eta'(y)|$ where $y>0$ is any value such that $\eta'(y)\neq0$. When $\Delta_{p}$ degenerates into a point $\nu$, $\eta(y_{\infty})=\nu$ and $\eta'(y_{\infty})=0$. Finally, in the case where $\Delta_{p}$ is a non-degenerate parabola, $\eta(y_{\infty})$ and $\eta'(y_{\infty})$ are not defined. The previous construction was first considered for corank $1$ surfaces in $\mathbb{R}^{3}$ (see \cite{Benedini/Sinha/Ruas,MartinsBallesteros}).

\subsection{Corank $1$ $3$-manifolds in $\mathbb{R}^5$}

In \cite{BenediniRuasSacramento}, the authors dedicate themselves to the study of singular corank $1$ $3$-manifolds in $\mathbb{R}^5$,
inspired by \cite{Benedini/Sinha/Ruas,MartinsBallesteros}. In that paper, they define the fundamental forms, the curvature locus and
also investigate some aspects of the second order geometry of those manifolds.

Let $M^3_{\sing}\subset\mathbb{R}^5$ be a 3-manifold with a singularity of corank 1 at $p\in M$. The construction here is the same as for singular surfaces. We assume that $M^3_{\sing}$ is the image of a smooth map $g:\tilde{M}\rightarrow\mathbb{R}^5$, where $\tilde{M}$ is a smooth regular 3-manifold and $q\in \tilde{M}$ is a
singular corank 1 point of $g$ such that $g(q)=p$.
Taking $\phi:U\rightarrow \mathbb{R}^3$ defined on some open neighbourhood $U$ of $q$ in $\tilde{M}$,
we say that $f=g\circ\phi^{-1}$ is a local parametrisation of $M^3_{\sing}$ at $p$.

The following definitions are analogous to the ones presented before: tangent space ($T_pM^3_{\sing}$), normal space ($N_pM^3_{\sing}$) and
first fundamental form, $I:T_q\tilde{M}\times T_q\tilde{M}\rightarrow\mathbb{R}$.
Taking the frame $\{\partial x,\partial y,\partial z\}$ of $T_{q}\tilde{M}$, the coefficients of the first fundamental form with respect to $\phi$ are:
$$
\begin{array}{cc}
    E_{xx}(q)=I(\partial x,\partial x)=\langle f_x,f_x\rangle(\phi(q)),  & E_{xy}(q)=I(\partial x,\partial y)=\langle f_x,f_y\rangle(\phi(q)), \\
    E_{yy}(q)=I(\partial y,\partial y)=\langle f_y,f_y\rangle(\phi(q)), & E_{zz}(q)=I(\partial z,\partial z)=\langle f_z,f_z\rangle(\phi(q)), \\
    E_{xz}(q)=I(\partial x,\partial z)=\langle f_x,f_z\rangle(\phi(q)), & E_{yz}(q)=I(\partial y,\partial z)=\langle f_y,f_z\rangle(\phi(q)).
\end{array}
$$
Notice that if $u=a\partial x+b\partial y + c\partial z$ then
$$I(u,u)=a^2E_{xx}(q)+2abE_{xy}(q)+b^2E_{yy}(q)+c^2E_{zz}(q)+2acE_{xz}(q)+2bcE_{yz}(q).$$

The second fundamental form of $M^3_{\sing}$ at $p$ is the map $I:T_{q}\tilde{M}\times T_{q}\tilde{M}\rightarrow N_{p}M^3_{\sing}$ given by
$$\begin{array}{ccc}
    II(\partial x,\partial x)= \pi_2(f_{xx}),\,\,&
    II(\partial x,\partial y)= \pi_2(f_{xy}),\,\,&
    II(\partial y,\partial y)= \pi_2(f_{yy}),  \\
    II(\partial z,\partial z)= \pi_2(f_{zz}),\,\,&
    II(\partial x,\partial z)= \pi_2(f_{xz}),\,\,&
    II(\partial y,\partial z)= \pi_2(f_{yz}),
  \end{array}
$$
where $\pi_2:T_p\mathbb{R}^4\rightarrow N_pM^3_{\sing}$ is the orthogonal projection and they are all evaluated in $\phi(q)$
and we extend $II$ to $T_{q}\tilde{M}\times T_{q}\tilde{M}$ in a unique way as a symmetric bilinear map.

Given a normal vector $\nu\in N_{p}M$, the second fundamental form of $M^3_{\sing}$ at $p$ along $\nu$, $II_\nu:T_{q}\tilde{M}\times T_{q}\tilde{M}\rightarrow \mathbb{R}$, is defined by $II_\nu(u,v)=\langle II(u,v),\nu\rangle$.

The coefficients of $II_\nu$ in terms of local coordinates $(x,y,z)$ are:
$$\begin{array}{ccc}
    l_\nu(q)=\langle \pi_2(f_{xx}),\nu\rangle,\,\, & m_\nu(q)= \langle \pi_2(f_{xy}),\nu\rangle,\,\,  & n_\nu(q)= \langle \pi_2(f_{yy}),\nu\rangle,  \\
    p_\nu(q)= \langle \pi_2(f_{zz}),\nu\rangle,\,\, & q_\nu(q)= \langle \pi_2(f_{xz}),\nu\rangle,\,\,  & r_\nu(q)= \langle \pi_2(f_{yz}),\nu\rangle,
  \end{array}
$$
and the partial derivatives are all evaluated at $\phi(q)$.

For a fixed orthonormal frame $\{\nu_1,\nu_2,\nu_3\}$ of $N_{p}M^3_{\sing}$, the quadratic form associated to the second fundamental form is
$$  II(u,u)= \sum_{i=1}^3II_{\nu_i}(u,u)\nu_i=\sum_{i=1}^3(a^2 l_{\nu_i}+2abm_{\nu_i}+b^2n_{\nu_i}+c^2p_{\nu_i}+2acq_{\nu_i}+2bcr_{\nu_i})\nu_i,
$$
and the above coefficients calculated at $q$. Furthermore, in terms of the chosen frame, the second fundamental form can be represented by the following $3\times6$ matrix of coefficients: $$\left(
                                                         \begin{array}{cccccc}
                                                           l_{\nu_1} & m_{\nu_1}& n_{\nu_1} & p_{\nu_1} & q_{\nu_1} & r_{\nu_1} \\
                                                           l_{\nu_2} & m_{\nu_2} & n_{\nu_2} & p_{\nu_2} & q_{\nu_2} & r_{\nu_2}\\
                                                           l_{\nu_3} & m_{\nu_3} & n_{\nu_3} & p_{\nu_3} & q_{\nu_3} & r_{\nu_3}
                                                         \end{array}
                                                       \right).
$$

Let $C_q$ be the subset of unit vectors of $T_q\tilde{M}$ and let $\eta:C_q\rightarrow N_pM^3_{\sing}$ be the map given by $\eta(u)=II(u,u).$ We define the \emph{curvature locus} of $M^3_{\sing}$ at $p$, which we shall denote by $\Delta_{cv}$, as the subset $\eta(C_q)$.

Using suitable change of coordinates and rotations, we can write
$$f(x,y,z)=(x,y,f_1(x,y,z),f_2(x,y,z),f_3(x,y,z)),$$
with $(f_i)_x=(f_i)_y=(f_i)_z=0$ at $\phi(q)$, for $i=1,2,3$. Hence, the coefficients of the first fundamental form are $E=G=1$ and $F=H=I=J=0$. Furthermore, given a unit tangent vector $u\in C_q$ and writing $u=x\partial x+y\partial y+z\partial z$, since $$x^2E_{xx}(q)+2xyE_{xy}(q)+y^2E_{yy}(q)+z^2E_{zz}(q)+2xzE_{xz}(q)+2yzE_{yz}(q)=1$$
we have $x^2+y^2=1$, that is, $C_q$ is a unit cylinder parallel to the $z$-axis. Fixing an orthonormal frame $\{\nu_1,\nu_2,\nu_3\}$ of $N_pM^3_{\sing}$,
$$(x,y,z)\mapsto\sum_{i=1}^{3}(x^2 l_{\nu_i}+2xym_{\nu_i}+y^2n_{\nu_i}+z^2p_{\nu_i}+2xzq_{\nu_i}+2yzr_{\nu_i})\nu_i$$
is a parametrisation for curvature locus $\Delta_{cv}$, where $x^2+y^2=1$.

Similarly to the results in  \cite{Benedini/Sinha/Ruas,MartinsBallesteros}, the authors in \cite{BenediniRuasSacramento} presented a partition of all corank 1 map germs $f:(\mathbb{R}^3,0)\rightarrow(\mathbb{R}^5,0)$ according to their 2-jet under the action of $\mathcal{A}^2$, which denotes the space of 2-jets of diffeomorphisms in source and target. We denote by $J^2(3,5)$ the subspace of 2-jets $j^2f(0)$ of map germs $f:(\mathbb{R}^3,0)\rightarrow(\mathbb{R}^5,0)$ and by $\Sigma^1J^2(3,5)$ the subset of 2-jets of corank 1.

\begin{prop}\label{prop:orbitas}
There exist six orbits in $\Sigma^1J^2(2,3)$ under the action of $\mathcal{A}^2$, which are $$(x,y,xz,yz,z^2),\,(x,y,z^2,xz,0),\,(x,y,xz,yz,0),\,(x,y,z^2,0,0),\,(x,y,xz,0,0),\,(x,y,0,0,0).$$
\end{prop}



\section{Normal sections}\label{sections}






Consider $M^3_{\reg}\subset\mathbb{R}^{3+k}$, $k\geq1$ a regular $3$-manifold (resp. $M^3_{\sing}\subset\mathbb{R}^5$ a singular corank $1$ $3$-manifold). Let $u$ be a tangent direction in $T_pM^3_{\reg}$ (resp. $T_pM^3_{\sing}$) and $\{u=0\}$ the hyperplane in $\mathbb{R}^{3+k}$ (resp. $\mathbb{R}^5$) orthogonal to $u$. The \emph{normal section} of $M^3_{\reg}$ along $u$ is a regular surface $M^2_{\reg}=M^3_{\reg}\cap\{u=0\}$ contained in $\mathbb{R}^{3+k}\cap\{u=0\}\cong\mathbb{R}^{2+k}$ (resp. the normal section $M^2_{\sing}$ along $u$ is a singular corank $1$ surface $M^2_{\sing}=M^3_{\sing}\cap\{u=0\}$ contained in $\mathbb{R}^{5}\cap\{u=0\}\cong\mathbb{R}^{4}$).

In view of this, one may ask whether there is a relation between the curvature locus of $M^3_{\reg}\subset\mathbb{R}^{3+k}$ at $p$ and the curvature ellipse of $M^2_{\reg}$ at $p$ (resp. the curvature locus of $M^3_{\sing}\subset\mathbb{R}^5$ at $p$ and the curvature parabola of $M^2_{\sing}$ at the same point). The answer to this questions is yes in both cases. Nevertheless, the cases will be treated separately, since the proof of the singular case is more delicate.

\begin{teo}\label{regsections}
Let $M^3_{\reg}\subset\mathbb{R}^{3+k}$, $k\geq1$, a regular $3$-manifold and $p\in M^3_{\reg}$. The curvature locus of $M^3_{\reg}$
at $p$ is generated by the union of the curvature ellipses at $p$ of the regular surfaces in $\mathbb{R}^{2+k}$ given by the normal sections along the tangent directions of $M^3_{\reg}$.
\end{teo}
\begin{proof}
Assume, without loss of generality, that $p$ is the origin. Take a parametrisation of $M^3_{\reg}$
in the Monge form
$f:(\mathbb{R}^3,0)\rightarrow(\mathbb{R}^{3+k},0)$, with
$$f(x,y,z)=(x,y,z,f_1(x,y,z),\ldots,f_k(x,y,z)),$$
and $f_i\in\mathcal{M}_{3}^2$, for $i=1,\ldots,k$. Take coordinates $(X,Y,Z,W_1,\ldots,W_k)$
in $\mathbb{R}^{3+k}$.

Let $u\in\mathbb{S}^2\subset T_pM^3_{\reg}$ be a non-zero vector and $\alpha_1,\alpha_{2},\alpha_3\in\mathbb{R}$ not all zero such that $u=\alpha_1X+\alpha_2Y+\alpha_3Z$ and $\alpha_1^2+\alpha_2^2+\alpha_3^2=1$. Consider the normal section given by the $(2+k)$-space generated by $\{\alpha_1X+\alpha_2Y+\alpha_3Z=0\}$. Suppose $\alpha_3\neq0$. Hence,
$Z=-\frac{\alpha_1}{\alpha_3}X-\frac{\alpha_2}{\alpha_3}Y=\beta_1X+\beta_2Y$,
where $\beta_1,\beta_2\in\mathbb{R}$.

The regular surface $M^2_{\reg}=M^3_{\reg}\cap\{Z=\beta_1X+\beta_2Y\}$ in $\mathbb{R}^{3+k}$ (although it is contained in a copy of $\mathbb{R}^{2+k}$) is locally given by
$$f(x,y)=(x,y,\beta_1x+\beta_2y,f_1(x,y,\beta_1x+\beta_2y),\ldots,f_k(x,y,\beta_1x+\beta_2y)).$$

The tangent plane $T_pM^2_{\reg}$ is such that its subset of unit vectors $\mathbb{S}^1$ is also a subset of $\mathbb{S}^2\subset T_{p}M^3_{\reg}$, since the curvature locus of $M^3_{\reg}$ is the image, via second fundamental form, of $\mathbb{S}^{2}$ and its restriction to $T_{p}M^2_{\reg}$ is precisely the second fundamental form of $M^2_{\reg}$ at $p$. Therefore the curvature ellipse of $M^2_{\reg}$ at $p$ is contained in the curvature locus of $M^3_{\reg}$ at $p$.

Finally, varying $u$ in $\mathbb{S}^2\subset T_pM^3_{\reg}$, we obtain all possible normal sections and the corresponding unit circles $\mathbb{S}^1$ cover the sphere $\mathbb{S}^2$. Hence, the curvature locus of $M^3_{\reg}$ is given by the union of the curvature ellipses.
\end{proof}

\begin{ex}
\begin{itemize}
    \item[(i)] Consider $M^3_{\reg}\subset\mathbb{R}^6$ given by $f:(\mathbb{R}^3,0)\rightarrow(\mathbb{R}^6,0)$, $$f(x,y,z)=\left(x,y,z,\frac{\sqrt{2}}{2}xy,\frac{\sqrt{2}}{2}xz,\frac{\sqrt{2}}{2}yz\right).$$ At the origin $p$, its curvature locus is a Roman Steiner surface.
    The normal sections given by $\{X=0\}$, $\{Y=0\}$ and $\{Z=0\}$, are regular surfaces whose curvature ellipses at $p$ are, respectively:
    $$
    \begin{array}{c}
    \eta_X(\theta)=(0,0,\sqrt{2}\sin(\theta)\cos(\theta)),\ \eta_Y(\theta)=(0,\sqrt{2}\sin(\theta)\cos(\theta),0),\\
    \eta_Z(\theta)=(\sqrt{2}\sin(\theta)\cos(\theta),0,0),
    \end{array}
    $$
    where $\theta\in[0,2\pi]$. In all the cases, the curvature ellipse is a segment which corresponds to the double point curve of the Roman Steiner surface. The normal sections
    $\{X=Y\}$, $\{X=Z\}$ and $\{Y=Z\}$, after changes of coordinates in the source and rotations in the tangent spaces of the surfaces in $\mathbb{R}^{5}$, provide us, respectively, the following curvature ellipses:
    $$
    \begin{array}{c}
         \eta_{XY}(\theta)=(\frac{\sqrt{2}}{2}\sin(\theta)^2,\sin(\theta)\cos(\theta),\sin(\theta)\cos(\theta)),\\ \eta_{XZ}(\theta)=(\sin(\theta)\cos(\theta),\frac{\sqrt{2}}{2}\sin(\theta)^2,\sin(\theta)\cos(\theta))  \\
          \eta_{YZ}(\theta)=(\sin(\theta)\cos(\theta),\sin(\theta)\cos(\theta),\frac{\sqrt{2}}{2}\sin(\theta)^2),
    \end{array}
    $$
    where $\theta\in[0,2\pi]$. This time, all curves are non degenerate ellipses. Figure \ref{steiner-folhas} shows the curvature ellipses on the Roman Steiner surface. It seems Steiner himself already knew how to generate the Roman surface by ellipses (see \cite{Apery}). However, all his ellipses pass through a ``pole'' whereas all of the ellipses obtained here pass through the triple point.
    \begin{figure}[h!]
\begin{center}
\includegraphics[scale=0.2]{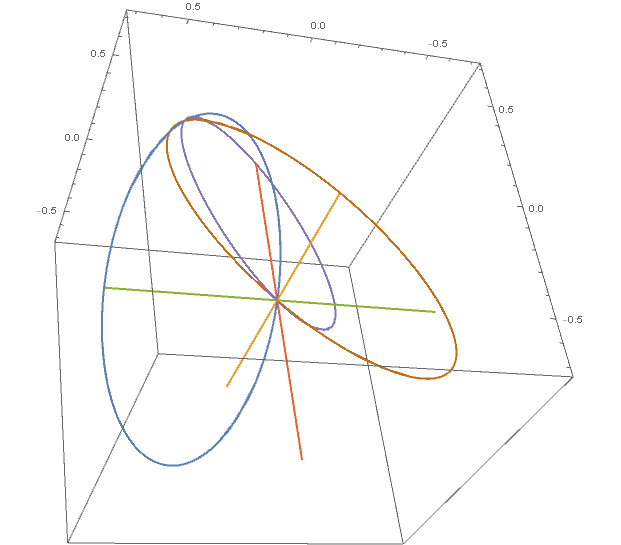}
\caption{Curvature ellipses on the Roman Steiner surface.}
\label{steiner-folhas}
\end{center}
\end{figure}
\item[(ii)] Consider $M^3_{\reg}\subset\mathbb{R}^5$ given by $f(x,y,z)=(x,y,z,x^2+z^2,xy)$. Taking coordinates $(X,Y,Z,W,T)$ in $\mathbb{R}^5$, its curvature locus at the origin $p$ is an elliptic region contained in the normal plane $\{W,T\}$, with center at $(1,0)$ and radius $1$.
\begin{table}[h]
\caption{Curvature ellipses on the elliptic region.}
\centering
{\begin{tabular}{clr}
\hline
Normal section & Parametrisation of the curvature ellipse & Type\\
\hline
$\{X=0\}$ & $(2\sin(\theta)^2,0)$  & segment \cr
$\{Y=0\}$ & $(2,0)$ & point \cr
$\{Z=0\}$ & $(2\sin(\theta)^2,2\sin(\theta)\cos(\theta))$ &  circle \cr
$\{X=Z\}$ & $(2\sin(\theta)^2,\frac{2}{\sqrt{2}}\sin(\theta)\cos(\theta))$ & ellipse \cr
$\{Y=Z\}$ & $(2\sin(\theta)^2+\cos^2(\theta),\frac{2}{\sqrt{2}}\sin(\theta)\cos(\theta))$ & ellipse \\
$\{X=Y\}$ & $(\sin(\theta)^2+2\cos(\theta)^2,\sin(\theta)^2)$ & segment \cr
\hline
\end{tabular}
}
\label{tab:folheacoes}
\end{table}
Table \ref{tab:folheacoes}, shows some curvature ellipses of regular surfaces given by normal sections. Here,
$\theta\in[0,2\pi]$. Figure \ref{elipses-folhas} shows the curves in Table \ref{tab:folheacoes}.
\begin{figure}[h!]
\begin{center}
\includegraphics[scale=0.3]{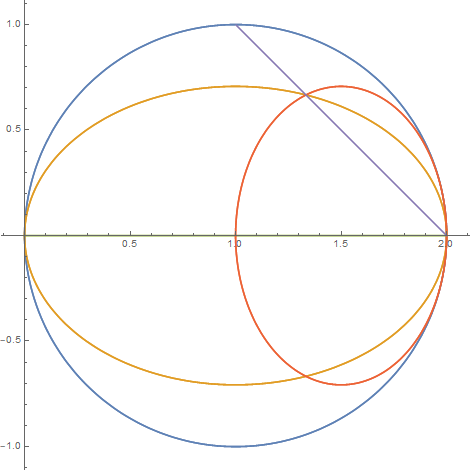}
\caption{Curvature ellipses on the elliptic region.}
\label{elipses-folhas}
\end{center}
\end{figure}
\end{itemize}
\end{ex}

Although it was known that the Roman Steiner surface could be generated by ellipses, geometrically speaking this is not so obvious for the cross-cap, the Steiner Type 5 or the cross-cup surface.

\begin{teo}\label{singsections}
Let $M^3_{\sing}\subset\mathbb{R}^{5}$ a singular corank $1$ $3$-manifold. The curvature locus of $M^3_{\sing}$
at $p$ is generated by the union of the curvature parabolas at $p$ of the singular surfaces in $\mathbb{R}^{4}$ given by the normal sections along the tangent directions of $M^3_{\sing}$.
\end{teo}
\begin{proof}
Consider $w\in T_pM^3_{\sing}$ a non zero vector. Here,
$(dg_q)^{-1}(w)\subset T_q\tilde{M}$ is a plane which contains the subset $\ker(dg_q)$, where $g$ is the corank $1$ map at $q$ used in the initial construction, where $g(q)=p$.

Hence, the subset $C'_q=(dg_q)^{-1}(w)\cap C_q$ is a pair of lines contained in the unit cylinder $C_q$ and such that $\eta_q(C'_q)$ is the curvature parabola at $p$ of the singular surface contained in the $4$-space given by the normal section $\{w=0\}$. Besides, the curvature parabola is a subset of the curvature locus of $M^3_{\sing}$. The second fundamental form of $M^3_{\sing}$ restricted to $(dg_q)^{-1}(w)\subset T_q\tilde{M}$ is precisely the second fundamental form of the singular surface $M^2_{\sing}\subset\mathbb{R}^4$. Figure \ref{fig.dem3var} shows the previous construction.
\begin{figure}[h!]
\begin{center}
\includegraphics[scale=0.4]{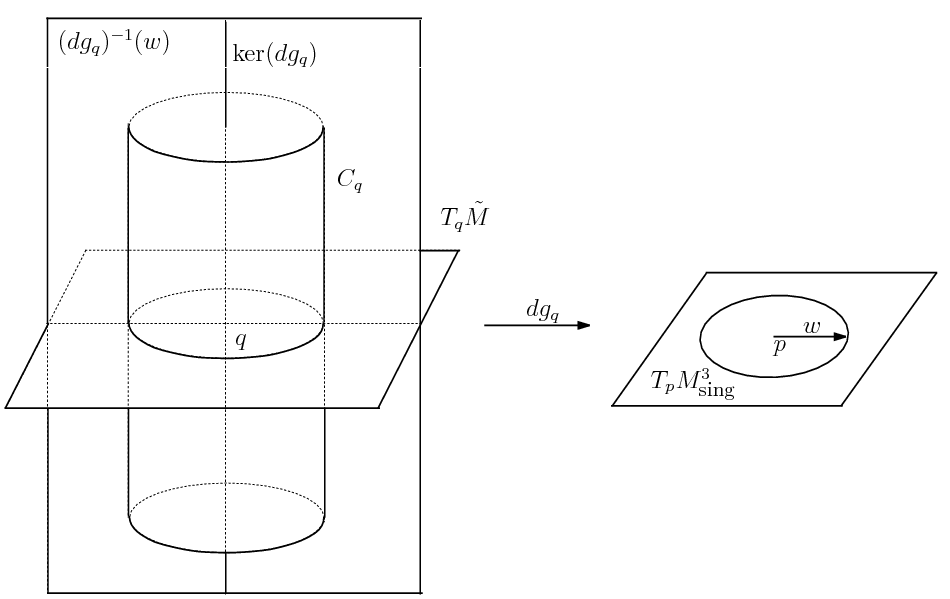}
\caption{Theorem \ref{singsections}.}
\label{fig.dem3var}
\end{center}
\end{figure}

Varying $w\in T_pM^3_{\sing}$, we obtain the cylinder $C_q$ in $T_q\tilde{M}$, therefore, the curvature locus of the $3$-manifold: since each normal section induces two lines which cover the cylinder when varying the normal section, the curvature locus of $M^3_{\sing}$ at $p$ is generated by the reunion of these curves.
\end{proof}

\begin{ex}\label{ex.sing}
\begin{itemize}
    \item[(i)] Let $M^3_{\sing}\subset\mathbb{R}^5$ be the singular $3$-manifold at the origin $p$ locally given by
    $f(x,y,z)=(x,y,x^{2}-2 yz,y^2-2 xz,z^2-2 xy)$
    whose curvature locus $\Delta_{cv}$ at $p$ is
    $$\{(2\alpha^2-4\beta\gamma,2\beta^2-4\alpha\gamma,2\gamma^2-4\alpha\beta)\in N_pM^3_{\sing}:\ \alpha^2+\beta^2=1\},$$
     \begin{figure}[!htb]
\begin{minipage}[b]{0.4\linewidth} \hspace{1.cm}\vspace{-0.5cm}
\includegraphics[width=\linewidth]{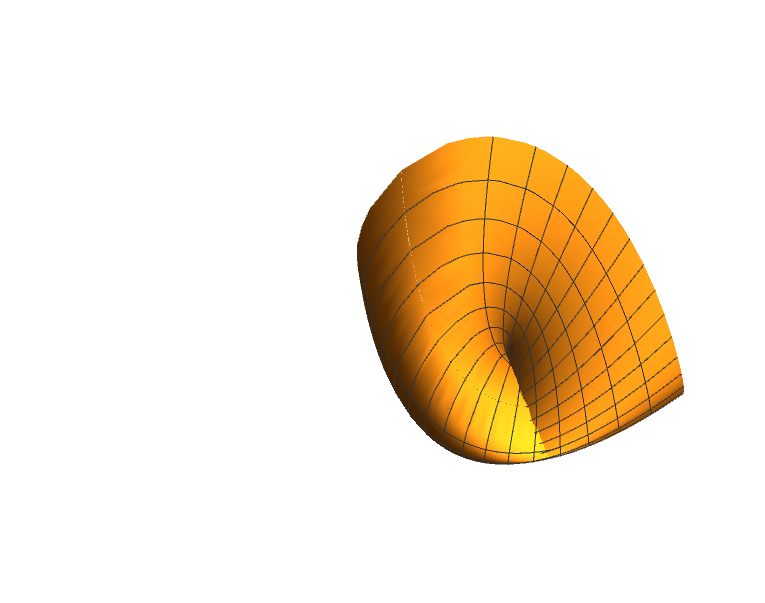}
\end{minipage} \hfill
\begin{minipage}[b]{0.4\linewidth} \hspace{-2.5cm}
\includegraphics[width=\linewidth]{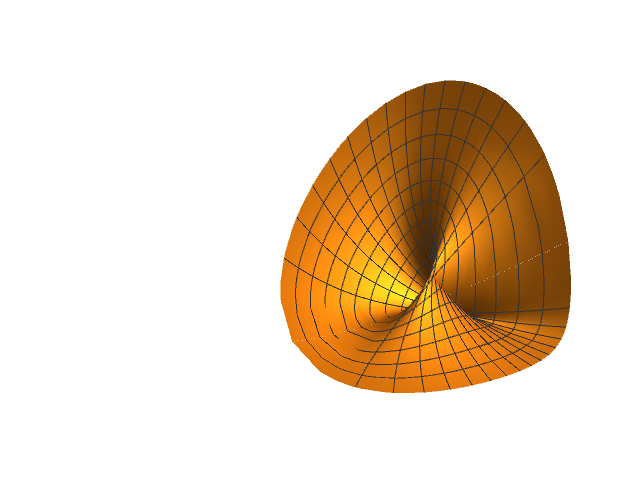}
\end{minipage}
\caption{Side and top views of $\Delta_{cv}$}
\label{Locus-sing}
\end{figure}
showed in Figure \ref{Locus-sing}.
    The normal section given by $\{X=0\}$, is the corank $1$ surface $\bar{f}(x,y)=(y,-2yz,y^2,z^2)$ and its curvature parabola at $p$ is  parametrised by $\bar{\eta}(z)=(-4z,2,2z^2)$. The normal section $\{Y=0\}$, parametrised by $\tilde{f}(x,z)=(x,x^2,-2xz,z^2)$ is such that its curvature parabola is also a non degenerated parabola,
    $\tilde{\eta}(z)=(2,-4z,2z^2)$. Taking the normal section $\{X+aY=0\}$, where $a\neq0$, after changes of coordinates in the source and isometries in the target, we obtain the singular surface given by
$$\left(0,y,\frac {  a^2\sqrt{a^2+1}y^2-2(a^2+1)yz }{
 \left( {a}^{2}+1 \right) ^{3/2}},\frac {
 \sqrt{a^2+1}y^2+2a(a^2+1)yz }{ \left( {a}^{2}
+1 \right) ^{3/2}},\frac {({a}^{2}+1){z}^{2}+2\,a{y}
^{2}}{{a}^{2}+1}\right),$$
and $\Delta_{cv}$ is parametrised by
$$\eta_a(z)=\left(\frac{2a^2\sqrt{a^2+1}-4(a^2+1)z}{( {a}^{2}+1)^{3/2}},\frac{2\sqrt{a^2+1}+4a(a^2+1)z}{( {a}^{2}+1)^{3/2}},\frac{4a+2(a^2+1)z^2}{a^2+1}\right),$$
a non degenerate parabola for $a\in\mathbb{R}$. Figure \ref{folhas1}, shows some of the curvature parabolas in the curvature locus.
\begin{figure}[h!]
\begin{center}
\includegraphics[scale=0.35]{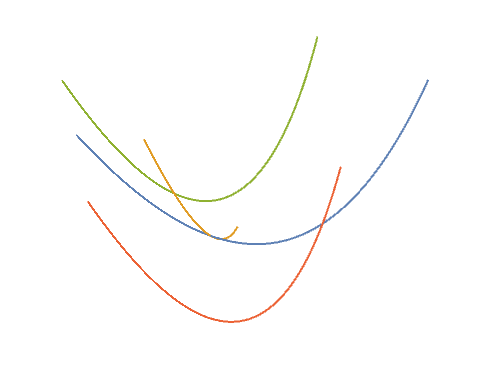}
\caption{Curvature parabolas.}
\label{folhas1}
\end{center}
\end{figure}
\item[(ii)] Let $M^3_{\sing}\subset\mathbb{R}^5$ be locally parametrised by $f(x,y,z)=(x,y,z^2,xz,0)$.
The curvature locus $\Delta_{cv}$ at the origin $p$ is the subset
$$\{(2\gamma^2,2\alpha\gamma,0)\in N_pM^3_{\sing};\ \alpha^2+\beta^2=1\},$$
a planar parabolic region, as in Figure \ref{regiao}.
\begin{figure}[h!]
\begin{center}
\includegraphics[scale=0.3]{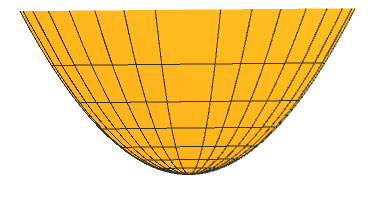}
\caption{Planar parabolic region.}
\label{regiao}
\end{center}
\end{figure}
The normal section given by $\{X=0\}$ is the corank $1$ surface
parametrised by $\bar{f}(y,z)=(y,z^2,0,0)$, whose curvature parabola is the half-line $\bar{\eta}(z)=(2z^2,0,0)$. The remaining normal sections given by $\{Y+aX=0\}$, where $a\in\mathbb{R}$ are corank $1$ surfaces parametrised by $\tilde{f}_a(x,z)=(x,-ax,z^2,xz,0)$, that can be written (after suitable change of coordinates, as before) as
$(x,z)\mapsto\left(0,x,z^2,\frac{\sqrt{a^2+1}}{a^2+1}xz\right)$.
The curvature parabolas of those surfaces are parametrised by $\eta_a(z)=(2z^2,\frac{2\sqrt{a^2+1}}{a^2+1}z,0)$, and for all $a\in\mathbb{R}$, their traces are non degenerate parabolas. 
\end{itemize}
\end{ex}

The curvature parabola's topological type of a corank $1$ surface $M^2_{\sing}\subset\mathbb{R}^n$, $n=3,4$ is a complete invariant
for the $\mathcal{A}^2$-classification of $2$-jets in $\Sigma^1J^2(2,n)$, as shown in \cite{Benedini/Sinha/Ruas,MartinsBallesteros}. The curvature locus of a corank $1$
$3$-manifold in $\mathbb{R}^5$ does not have the same property: the curvature locus of the $3$-manifold given by $g(x,y,z)=(x,y,xz,yz,z^2)$ at the origin $p$ is the paraboloid $\Delta_p=\{(0,0,2ac,2bc,2c^2)\mid a^2+b^2=1\}$, but the $3$-manifold given by $f(x,y,z)=(x,y,xz+y^2,yz,z^2)$, which satisfies $j^2f(0)\sim_{\mathcal{A}^2} (x,y,xz,yz,z^2)$ has the curvature locus at the origin $\Delta_{cv}=\{(0,0,2b^2+2ac,2bc,2c^2)\,|\,a^2+b^2=1\}$, which is not a paraboloid, as shown in \cite{BenediniRuasSacramento}.
However, the topological type of the curvature parabolas
of the normal sections gives necessary conditions for the $\mathcal{A}^2$-orbits of the $3$-manifold's parametrisation.

\begin{teo}\label{necessaryorbits}
Let $M^3_{\sing}\subset\mathbb{R}^{5}$ be a corank $1$ $3$-manifold at
$p\in M^3_{\sing}$. We assume $p$ the origin and denote by $j^{2}f(0)$ the $2$-jet of a local
parametrisation $f:(\mathbb{R}^{3},0)\rightarrow(\mathbb{R}^{5},0)$ of $M^3_{\sing}$. The following holds:
\begin{itemize}
\item[(i)] $j^{2}f(0)\sim_{\mathcal{A}^{2}}(x,y,xz,yz,z^2)\Rightarrow\Delta_{cv}$ is generated exclusively by non degenerate parabolas;
\item[(ii)] $j^{2}f(0)\sim_{\mathcal{A}^{2}}(x,y,z^2,xz,0)\Rightarrow\Delta_{cv}$ is generated by non degenerate parabolas and a half-line;
\item[(iii)] $j^{2}f(0)\sim_{\mathcal{A}^{2}}(x,y,xz,yz,0)\Rightarrow\Delta_{cv}$ is generated exclusively by lines;
\item[(iv)] $j^{2}f(0)\sim_{\mathcal{A}^{2}}(x,y,z^2,0,0)\Rightarrow\Delta_{cv}$ is generated exclusively by half-lines;
\item[(v)] $j^{2}f(0)\sim_{\mathcal{A}^{2}}(x,y,xz,0,0)\Rightarrow\Delta_{cv}$ is generated by lines and a point;
\item[(vi)] $j^{2}f(0)\sim_{\mathcal{A}^{2}}(x,y,0,0,0)\Rightarrow\Delta_{cv}$ is generated exclusively by points.
\end{itemize}
\end{teo}
\begin{proof} Since the proofs of all cases are similar, we shall present only the first case.
In \cite{BenediniRuasSacramento}, the authors proved that if $j^{2}f(0)\sim_{\mathcal{A}^{2}}(x,y,xz,yz,z^2)$, then
$f$ is $\mathcal{R}^2\times\mathcal{O}(5)$-equivalent to
    $$
\begin{array}{cl}
   (x,y,z)\mapsto  & (x,y,a_1x^2+a_3y^2+xz+a_6yz,b_1x^2+b_2xy+b_3y^3+b_6yz, \\
     & c_1x^2+c_2xy+c_3y^2+c_4z^2+c_5xz+c_6yz),
\end{array}
$$
where $c_4>0$ and $b_6\neq0$. Here, $\mathcal{R}^2$ denotes the group of $2$-jets of diffeomorphisms from $(\mathbb{R}^3,0)$ to $(\mathbb{R}^3,0)$ and $\mathcal{O}(5)$ is the group of linear isometries of $\mathbb{R}^5$.

Consider the normal section given by $\{Y+\alpha X=0\}$, where $\alpha\in\mathbb{R}-\{0\}$, locally parametrised by
$$
\begin{array}{cl}
    (x,z)\mapsto & (x,-\alpha x,a_{3}\,{\alpha}^{2}{x}^{2}-a_{6}\,\alpha\,xz+a_{1
}\,{x}^{2}+xz,-\alpha\,b_{2}\,{x}^{2}+\alpha\,b_{3
}\,{x}^{2}-\alpha\,b_{6}\,xz+b_{1}\,{x}^{2},  \\
     & c_{3}
\,{\alpha}^{2}{x}^{2}-c_{2}\,\alpha\,{x}^{2}-c_{6}\,\alpha\,xz+c_{1}\,
{x}^{2}+c_{4}\,{z}^{2}+c_{5}\,xz)
\end{array}
$$
By a rotation of angle $\theta=\arctan(\alpha)$ in the target
and the change of coordinates in the source,
$(x,z)\mapsto(\frac{\sqrt{\alpha^2+1}x}{\alpha^2+1},z)$ we obtain,
$$
\begin{array}{l}
    (x,z)\mapsto \left(x,0,\frac{\sqrt{\alpha^2+1}(\alpha^2 a_3+a_1)x^2+(\alpha^2+1)(1-a_6\alpha)xz}{(\alpha^2+1)^{3/2}},-\frac{\sqrt{\alpha^2+1}(b_2\alpha-b_3\alpha-b_1)x^2+b_6\alpha(\alpha^2+1)xz}{(\alpha^2+1)^{3/2}}\right.,  \\
      \left.\frac{\sqrt{\alpha^2+1}(c_3\alpha^2-c_2\alpha+c_1)x^2+(\alpha^2+1)(c_5-\alpha c_6)xz+c_4\sqrt{\alpha^2+1}(\alpha^2+1)z^2}{(\alpha^2+1)^{3/2}}\right).
\end{array}
$$
The parametrisation of the normal section in the $4$-space $XZWT$, is such that its $2$-jet is $\mathcal{A}^2$-equivalent to $(x,xz,z^2,0)$, since the coefficient of $z^2$ is not zero. Hence, by Theorem 3.6 in \cite{Benedini/Sinha/Ruas}, the curvature parabola of the normal section is a non degenerate parabola for all $\alpha\neq0$. Finally, the normal sections given by $\{X=0\}$ and $\{Y=0\}$ are singular surfaces parametrised, respectively, by
$$(y,z)\mapsto(0,y,a_{3}y^{2}+a_{6}yz,b_{3}{y}^{2}+b_{6}yz,c_{3}y^{2}+c_{4}z^{2}+c_{6}yz),$$
and
$$(x,z)\mapsto(x,0,a_{1}x^{2}+xz,b_{1}x^{2},c_{1}x^{2}+c_{4}z^{2}+c_{5}xz),$$
and the $2$-jets of both of them are $\mathcal{A}^2$-equivalent to $(x,xy,y^2,0)$. Once again, by Theorem 3.6 in \cite{Benedini/Sinha/Ruas}, the
curvature parabolas are non degenerate parabolas. Therefore, $\Delta_{cv}$ is obtained exclusively by non degenerate parabolas.
\end{proof}

The converse of Theorem \ref{necessaryorbits}, nevertheless, is not true. The curvature locus of $M^3_{\sing}$ given by $f(x,y,z)=(x,y,z^2,xz,0)$ at the origin $p$, as in Example \ref{ex.sing}, is a planar region that can be seen as the union of only non degenerate parabolas.


\section{Relating second order geometry through projections and normal sections}\label{diagram}

When projecting a regular $n$-manifold in $\mathbb R^{n+k}$ along a tangent direction we obtain a singular $n$-manifold in $\mathbb R^{n+k-1}$. It is natural to expect certain relations between the curvature loci of each case. For example, in \cite{BenediniOset} we showed the relation between the curvature ellipse of $M^2_{\reg}\subset\mathbb R^4$ and the curvature parabola of the projection $M^2_{\sing}\subset\mathbb R^3$ and obtained some relations between their second order geometry. It is also known that in the previous case, the tangent direction is asymptotic if and only if the singularity of the projection is worse than a crosscap (\cite{mondthesis, BruceNogueira}). Similarly, for projections from $M^2_{\reg}\subset\mathbb R^5$ to $M^2_{\sing}\subset\mathbb R^4$, the direction is asymptotic if and only if the singularity is worse than an $I_1$-singularity (\cite{Fuster/Ruas/Tari}). In the next section we will show an equivalent result for projections from $M^3_{\reg}\subset\mathbb R^6$ to $M^3_{\sing}\subset\mathbb R^5$. We will also define asymptotic directions for $M^3_{\sing}\subset\mathbb R^5$ and relate them to the asymptotic directions of $M^3_{\reg}\subset\mathbb R^6$ and $M^2_{\sing}\subset\mathbb R^4$. However, first we will justify why these type of relations are possible.

As seen in the previous section, geometrical relations between manifolds are obtained not only by projections, but also by normal sections. In order to relate both these concepts we must consider projections in a tangent direction contained in the normal section. For simplicity we fix the direction of projection and the normal section.

\begin{teo}\label{commutative}
Let $M^3_{\reg}\subset \mathbb R^6$ be given in Monge form by
$$(x,y,z)\mapsto(x,y,z,f_1(x,y,z),f_2(x,y,z),f_3(x,y,z)),$$ let $v=(0,0,1)\in T_p M^3_{\reg}$ and let $\pi_v$ be the projection along the direction $v$. Consider the normal section given by $\{Y=0\}$. Let $i_1, i_2$ be the immersions of the normal sections in $\mathbb R^6$ and $\mathbb R^5$ respectively. Let $v'=i_{1_*}^{-1}(v)=(0,1)\in T _{i_1^{-1}(p)}M^2_{reg}$. We have a commutative diagram
\begin{equation*}
\begin{CD}
M^3_{\reg}\subset\mathbb R^6 @>{\pi_v}>> M^3_{\sing}\subset\mathbb R^5\\ @A{i_1}AA        @AA{i_2}A   \\
M^2_{\reg}\subset\mathbb R^5 @>>{\pi_{v'}}> M^2_{\sing}\subset\mathbb R^4
\end{CD}
\end{equation*}
where $M^2_{\reg}=M^3_{\reg}\cap\{Y=0\}$ and $M^3_{\sing}, M^2_{\sing}$ are the corresponding singular projections, which induces a commutative diagram amongst the curvature loci of the four manifolds.
\end{teo}
\begin{proof}
Consider $(X,Y,Z,W,T,S)$ to be the coordinates of $\mathbb R^6$, then $i_1$ and $i_2$ are given by $i_1(X,Z,W,T,S)=(X,0,Z,W,T,S)$ and $i_2(X,W,T,S)=(X,0,W,T,S)$. $M^2_{\reg}$ is given by $(x,z,f_1(x,0,z),f_2(x,0,z),f_3(x,0,z))$ and clearly $\pi_v\circ i_1(M^2_{\reg})=i_2\circ \pi_{v'}(M^2_{\reg})$.

Now, the curvature locus of $M^3_{\reg}$ is the image by $II$ of the unit tangent vectors in $T_p M^3_{\reg}$. We can parameterise the sphere $\mathbb{S}^2$ of unit tangent vectors in spherical coordinates by $(\theta,\phi)$, where $\theta\in[0,2\pi]$ is the azimuth (i.e. the angle from the $X$-axis in a plane of constant height) and $\phi\in[0,\pi]$ is the polar angle (i.e. the angle from the $Z$-axis). When projecting along the tangent direction $v=(0,0,1)$ we obtain a singular 3-manifold and instead of a metric we have a pseudo-metric in the tangent space. The unit tangent vectors in $T_{\pi_{v}(p)} \tilde{M}^3_{\sing}$ form a cylinder $C$ which is obtained by blowing up the north and south poles of $\mathbb{S}^2$. There is a natural map from $\mathbb{S}^2$ to $C$ which takes the spherical coordinates $(\sin(\phi)\cos(\theta),\sin(\phi)\sin(\theta),\cos(\phi))$ to the cylindrical coordinates $(\cos(\theta),\sin(\theta),\frac{\cos(\phi)}{\sin(\phi)})$ by dividing each component by $\sin(\phi)$ (i.e. it maps the point of intersection with $\mathbb{S}^2$ of a ray from the origin to the point of intersection with $C$, the north and south poles go to infinty). This map induces a relation between the parameterisations of the curvature locus of $M^3_{\reg}$ and $M^3_{\sing}$. In fact, since the loci are the image of $II$ and the coefficients of these second fundamental forms are the same in the regular and singular cases, the fact of $II$ being a quadratic homogeneous map means that if $\eta_e(\theta,\phi)$ is the parametrisation of the curvature locus of $M^3_{\reg}$, then $$\eta_p(\theta,\phi)=\frac{1}{\sin(\phi)^2}\eta_e(\theta,\phi)$$ is the parametrisation of the curvature locus of $M^3_{\sing}$.

On the other hand, the section $\{Y=0\}$ induces a section in $T_p M^3_{\reg}$. In spherical coordinates, this gives the section $\{\theta=0\}$ of $\mathbb{S}^2$. So, by Theorem \ref{regsections} the curvature ellipse of $M^2_{\reg}$ is given by $\eta_e(0,\phi)$. Similarly, by Theorem \ref{singsections} the curvature parabola of $M^2_{\sing}$ is given by $\eta_p(0,\phi)$.

It remains to see that to pass from the curvature ellipse to the curvature parabola we must divide each component of the parametrisation by $\sin(\phi)^2$. This follows from the geometrical interpretation of $\cot(\phi)=\frac{\cos(\phi)}{\sin(\phi)}$, which again shows that we must divide the components $(\sin(\phi),\cos(\phi))$ of $\mathbb{S}^1$ by $\sin(\phi)$ to get the components of the unit tangent vectors in $T_p \tilde{M}^2_{\sing}$, and the fact that the second fundamental form is a homogeneous quadratic map.
\end{proof}

\begin{rem}
In the proof of Proposition 3.8 in \cite{BenediniOset} in order to obtain the parametrisation of the curvature parabola of $M^2_{\sing}\subset \mathbb R^4$ from the parametrisation of the curvature ellipse of $M^2_{\reg}\subset \mathbb R^3$ we divide by $\cos(\phi)^2$ instead of $\sin(\phi)^2$. This is due to the fact that in the proof above, when we take the section $\{Y=0\}=\{\theta=0\}$, we are left with the $\{X,Z\}$-plane and the angle $\phi$ goes from the $Z$-axis to the $X$-axis, while the angle in the proof of Proposition 3.8 in \cite{BenediniOset} goes from the $X$-axis to the $Z$-axis
\end{rem}

\begin{ex}
Consider $M^3_{\reg}$ given by $f(x,y,z)=(x,y,z,x^2+\frac{1}{2}z^2,xz,yz)$. The projection along the tangent vector $(0,0,1)$ is $M^3_{\sing}$ given by $(x,y,x^2+\frac{1}{2}z^2,xz,yz)$, and the normal section $\{Y=0\}$ gives the regular surface $M^2_{\reg}$ given by $(x,z,x^2+\frac{1}{2}z^2,xz,0)$. The normal section of $M^3_{\sing}$, which coincides with the projection of $M^2_{\reg}$ along the tangent vector $(0,1)$, is given by $(x,x^2+\frac{1}{2}z^2,xz,0)$. The curvature locus of $M^3_{\reg}$ is a Steiner Roman surface parameterised by $$\eta_e(\theta,\phi)=(1+\sin(\phi)^2\cos(2\theta),\cos(\theta)\sin(2\phi),\sin(\theta)\sin(2\phi)),$$ and the curvature locus of $M^3_{\sing}$ is a Cylindrical Steiner surface given by $\eta_p(\theta,\phi)=$ $$\frac{1}{\sin(\theta)^2}(1+\sin(\phi)^2\cos(2\theta),\cos(\theta)\sin(2\phi),\sin(\theta)\sin(2\phi))=(2a^2+c^2,2ac,2bc),$$ where $a=\cos(\theta), b=\sin(\theta)$ and $c=\frac{\cos(\phi)}{\sin(\phi)}$, so $a^2+b^2=1$. The normal section of $M^3_{\sing}$ is given by $\{\theta=0\}=\{a=1,b=0\}$, so we get a curvature parabola $(2+c^2,2c,0)$.

On the other hand, the curvature ellipse of $M^2_{\reg}$ is parameterised by $(1+\sin(\phi)^2,\sin(2\phi),0)$ and dividing by $\sin(\phi)^2$ and changing $\frac{\cos(\phi)}{\sin(\phi)}=c$ we again obtain the curvature parabola $(2+c^2,2c,0)$.
\end{ex}

\begin{rem}
In some cases as the above example, the normal sections $M^2_{\reg}$ and $M^2_{\sing}$ can be seen in $\mathbb R^4$ and $\mathbb R^3$ respectively. In such cases we can add a line
\begin{equation*}
\begin{CD}
M^2_{\reg}\subset\mathbb R^4 @>>{\pi_{v'}}> M^2_{\sing}\subset\mathbb R^3
\end{CD}
\end{equation*} to the commutative diagram in Theorem \ref{commutative} with the corresponding immersions.
\end{rem}

Since the curvature locus contains all the second order geometry of the manifold we get

\begin{coro}
The second order geometries of $M^3_{\reg}\subset \mathbb R^6, M^2_{\reg}\subset \mathbb R^5, M^3_{\sing}\subset \mathbb R^5$ and  $M^2_{\sing}\subset \mathbb R^4$ are all related amongst each other.
\end{coro}

Theorem \ref{necessaryorbits} is an example of this fact. We study in what ways the geometry is related in the next subsection.

\section{Asymptotic direction of singular 3-manifolds in $\mathbb R^5$}\label{asymptotic}

Let $M$ be a corank 1 singular 3-manifold in $\mathbb R^5$, $p\in M$ and take $M$ as the image of a smooth map $g:\tilde{M}\rightarrow \mathbb{R}^{5}$, where $\tilde{M}$ is a regular 3-manifold and $q\in\tilde{M}$ is a corank $1$ point of $g$ such that $g(q)=p$.

\begin{definition}
A direction $u\in T_{q}\tilde{M}$ is called \emph{asymptotic} if there is a non zero vector $\nu\in N_pM$ such that
$$II_{\nu}(u,v)=\langle II(u,v),\nu\rangle=0\ \ \forall\ v\in T_{q}\tilde{M}.$$
Moreover, in such case, we say that $\nu$ is a \emph{binormal direction}.
\end{definition}

Following Theorem \ref{Teo-Dreibelbis}, in the regular case there are many ways of defining asymptotic directions and all of them are equivalent. We will prove a similar result for the singular case, but before proving this we need the following definition due to Dreibelbis (adapted for the singular case here).

\begin{definition}
Let $\{e_1,e_2,e_3\}$ be a basis for $T_q\tilde M$ and $\{n_1,n_2,n_3\}$ be a basis for $N_pM$. For any vector $u\in T_q\tilde M$, define $A(u)$ as the $3\times 3$ matrix with $A(u)_{ij}=II_{n_i}(e_j,u)=\langle II(e_j,u),n_i\rangle $.
\end{definition}

\begin{teo}\label{equiv}
Given $u\in T_q\tilde M$, the following are equivalent:
\begin{itemize}
\item[(1)] $u$ is an asymptotic direction.
\item[(2)] $\det A(u)=0$.
\item[(3)] There exists $\nu\in N_pM$ such that the height function $h_{\nu}$ has a degenerate singularity and $u\in\ker \Hess h_{\nu}$.
\end{itemize}
\end{teo}
\begin{proof}
Let $u=\alpha\partial_x+\beta\partial_y+\gamma\partial_z$, $v=\bar{\alpha}\partial_x+\bar{\beta}\partial_y+\bar{\gamma}\partial_z$ and $\nu=\nu_1n_1+\nu_2n_2+\nu_3n_3$.

$$\begin{array}{ll}
II(u,v) & =II(\alpha\partial_{x}+\beta\partial_{y}+\gamma\partial_z,\bar{\alpha}\partial_{x}+\bar{\beta}\partial_{y}+\bar{\gamma}\partial_z) \\
        & =\alpha\bar{\alpha}II(\partial_{x},\partial_{x})+(\alpha\bar{\beta}+\beta\bar{\alpha})II(\partial_{x},\partial_{y})+\beta\bar{\beta}II(\partial_{y},\partial_{y})+\gamma\bar{\gamma}II(\partial_{z},\partial_{z}) \\
        &\ +(\alpha\bar{\gamma}+\gamma\bar{\alpha})II(\partial_{x},\partial_{z})+(\beta\bar{\gamma}+\gamma\bar{\beta})II(\partial_{y},\partial_{z})
\end{array}
$$
and $II_{\nu}(u,v)=$
$$\begin{array}{l}
=\langle II(u,v),\nu_1n_1+\nu_2n_2+\nu_3n_3\rangle=\nu_1\langle II(u,v),n_{1}\rangle+\nu_2\langle II(u,v),n_{2}\rangle+\nu_3\langle II(u,v),n_{3}\rangle \\
              =\displaystyle\sum_{i=1}^{3}\nu_i[\alpha\bar{\alpha}l_{n_{i}}+(\alpha\bar{\beta}+\beta\bar{\alpha})m_{n_{i}}+\beta\bar{\beta}n_{n_{i}}+\gamma\bar{\gamma}p_{n_{i}}
        +(\alpha\bar{\gamma}+\gamma\bar{\alpha})q_{n_{i}}+(\beta\bar{\gamma}+\gamma\bar{\beta})r_{n_{i}}].
\end{array}
$$
Rewriting:
$$\begin{array}{c}
 =\bar{\alpha}[\displaystyle\sum_{i=1}^{3}\nu_i(\alpha l_{n_i}+\beta m_{n_i}+\gamma q_{n_i})]
  +\bar{\beta}[\displaystyle\sum_{i=1}^{3}\nu_i(\alpha m_{n_i}+\beta n_{n_i}+\gamma r_{n_i})] \\
  +\bar{\gamma}[\displaystyle\sum_{i=1}^{3}\nu_i(\alpha p_{n_i}+\beta q_{n_i}+\gamma r_{n_i})] .
\end{array}
$$
In order for $u\in T_{q}\tilde{M}$ to be an asymptotic direction, we must show that $II_{\nu}(u,v)=0$.
The last equality above must be satisfied for all $v=(\bar{\alpha},\bar{\beta},\bar{\gamma})\in T_{q}\tilde{M}$, so
$$\left\{\begin{array}{l}
\displaystyle\sum_{i=1}^{3}\nu_i(\alpha l_{n_i}+\beta m_{n_i}+\gamma q_{n_i})=0 \\
\displaystyle\sum_{i=1}^{3}\nu_i(\alpha m_{n_i}+\beta n_{n_i}+\gamma r_{n_i})=0 \\
\displaystyle\sum_{i=1}^{3}\nu_i(\alpha p_{n_i}+\beta q_{n_i}+\gamma r_{n_i})=0.
\end{array}\right.
$$
Since we want different solutions to $\nu_1=\nu_2=\nu_3=0$, we have
$$
\det\left(
      \begin{array}{ccc}
        \alpha l_{n_1}+\beta m_{n_1}+\gamma q_{n_1} & \alpha l_{n_2}+\beta m_{n_2}+\gamma q_{n_2} & \alpha l_{n_3}+\beta m_{n_3}+\gamma q_{n_3} \\
        \alpha m_{n_1}+\beta n_{n_1}+\gamma r_{n_1} & \alpha m_{n_2}+\beta n_{n_2}+\gamma r_{n_2} & \alpha m_{n_3}+\beta n_{n_3}+\gamma r_{n_3} \\
        \alpha p_{n_1}+\beta q_{n_1}+\gamma r_{n_1} & \alpha p_{n_2}+\beta q_{n_2}+\gamma r_{n_2} & \alpha p_{n_3}+\beta q_{n_3}+\gamma r_{n_3}
      \end{array}
    \right)=0,
$$
that is, $\det A(u)=0$. This proves the first equivalence.

For the second equivalence, we rewrite the above system of equations as
$$\left\{\begin{array}{l}
\alpha\displaystyle\sum_{i=1}^{3}\nu_i l_{n_i}+\beta\displaystyle\sum_{i=1}^{3}\nu_i m_{n_i}+\gamma\displaystyle\sum_{i=1}^{3}\nu_i q_{n_i}=0 \\
\alpha\displaystyle\sum_{i=1}^{3}\nu_i m_{n_i}+\beta\sum_{i=1}^{3}\nu_i n_{n_i}+\gamma\sum_{i=1}^{3}\nu_i r_{n_i}=0 \\
\alpha\displaystyle\sum_{i=1}^{3}\nu_i p_{n_i}+\beta\sum_{i=1}^{3}\nu_i q_{n_i}+\gamma\sum_{i=1}^{3}\nu_i r_{n_i}=0,
\end{array}\right.
$$
which means that $u\in \ker\Hess h_{\nu}$ and in order to get a different solution to $\alpha=\beta=\gamma=0$ we need $\det\Hess h_v=0$. This proves the third equivalence.
\end{proof}

\begin{teo}\label{asympregsing}
Let $M^3_{\sing}\subset \mathbb R^5$ be the projection under a tangent direction of $M^3_{\reg}\subset \mathbb R^6$ and $p'\in M^3_{\reg}$ the point which is projected to $p$. Then $u\in T_{p'}M^3_{\reg}$ is asymptotic if and only if $u\in T_q\tilde M$ is asymptotic. Moreover, the binormal directions are also the same.
\end{teo}
\begin{proof}
The coefficients of the second fundamental form are the same for the regular and the singular case, so the height functions are the same. By Theorem \ref{Teo-Dreibelbis} and (3) in Theorem \ref{equiv} we get the equivalence.
\end{proof}

\begin{prop}\label{projnull}
Let $M^n_{\sing}\subset \mathbb R^{n+k-1}$ be the projection under a tangent direction $u$ of $M^n_{\reg}\subset \mathbb R^{n+k}$. The direction of projection $u$ becomes the null tangent direction of the singular projection.
\end{prop}
\begin{proof}
Let $f(x_1,\ldots,x_n)$ be the parametrisation of $M^n_{\reg}$ and denote by $f_i:=\frac{\partial f}{\partial x_i}$ and by $E_{ij}:=\langle f_i,f_j\rangle$ the coefficients of the first fundamental form. Consider a unit tangent vector $u=\sum_{i=1}^na_if_i\in T_pM^n_{\reg}$, then $I(u,u)=\sum_{i=1}^na_i^2E_{ii}+2\sum_{1\leq i<j\leq n}a_ia_jE_{ij}=1$. Consider the projection in the direction $u$, $P_u=f-\langle f, u\rangle u$. The coefficients of the first fundamental form for the singular projection are $E^P_{ii}=\langle P_{u_i}, P_{u_i}\rangle =\langle f_i-\langle f_i, u\rangle u,f_i-\langle f_i, u\rangle u\rangle=E_{ii}-(\langle f_i, u\rangle)^2$ and similarly $E^P_{ij}=E_{ij}-\langle f_i, u\rangle\langle f_j, u\rangle$. So the first fundamental form of the singular projection applied to $u$ is
$$
\begin{array}{ll}
    I^P(u,u) & =\displaystyle\sum_{i=1}^na_i^2E^P_{ii}+2\sum_{1\leq i<j\leq n}a_ia_jE^P_{ij} \\
     & = \displaystyle\sum_{i=1}^na_i^2(E_{ii}-(\langle f_i, u\rangle)^2)+2\sum_{1\leq i<j\leq n}a_ia_j(E_{ij}-\langle f_i, u\rangle\langle f_j, u\rangle) \\
     & = \displaystyle I(u,u)-(\sum_{i=1}^na_i^2(\langle f_i, u\rangle)^2+2\sum_{1\leq i<j\leq n}a_ia_j\langle f_i, u\rangle\langle f_j, u\rangle) \\
     & = \displaystyle I(u,u)-(\sum_{i=1}^na_i\langle f_i, u\rangle)^2.\\
\end{array}
$$
On the other hand $\langle f_i, u\rangle=\langle f_i, \sum_{j=1}^na_jf_j\rangle=\sum_{j=1}^na_jE_{ij}$, so the above equation is equal to $$=I(u,u)-(\sum_{i=1}^na_i\sum_{j=1}^na_jE_{ij})^2=1-(\sum_{i=1}^na_i^2E_{ii}+2\sum_{1\leq i<j\leq n}a_ia_jE_{ij})^2=1-1=0.$$ So $u$ is the null tangent direction in the pseudo-metric of the singular projection.
\end{proof}

\begin{coro}\label{nullasymp}
The direction of projection is asymptotic if and only if the null tangent direction is asymptotic.
\end{coro}
\begin{proof}
Follows directly from Theorem \ref{asympregsing} and Proposition \ref{projnull}.
\end{proof}

\begin{definition}
When the null tangent direction $u\in T_pM^n_{\sing}$ is asymptotic we call it \emph{infinite asymptotic direction} and denote it by $u_{\infty}$.
\end{definition}

For $M^2_{\sing}\subset \mathbb R^3$ and $M^2_{\sing}\subset \mathbb R^4$ the null tangent direction $u$ is asymptotic only when the curvature parabola $\Delta_p$ is degenerate (see \cite{MartinsBallesteros} for $\mathbb R^3$ and \cite{Benedini/Sinha/Ruas} for $\mathbb R^4$). Since the curvature parabola $\Delta_p$ is the image by $\eta=II$ of $C_q$ and must contain all the second order geometry, the image $\eta(u_{\infty})$ and tangent space to $\Delta_p$ at $\eta(u_{\infty})$ must be defined. For $M^3_{\sing}\subset \mathbb R^5$ we also need to define this image and tangent space, however, not only for the degenerate case.

The idea of an infinite asymptotic direction is as follows. Any singular manifold can be seen as the projection of a regular manifold along a tangent direction. By Theorem \ref{asympregsing} the number of asymptotic directions in the regular and in the singular case is the same. When we project the regular manifold in an asymptotic direction, we force that asymptotic direction to become an infinite asymptotic direction of the singular projection. In fact, by Corollary \ref{nullasymp} it is the null tangent direction.

In the case of $M^2_{\reg}\subset \mathbb R^4$ projected to $M^2_{\sing}\subset \mathbb R^3$ the following are equivalent:
\begin{itemize}
\item[(i)] The direction of projection is an asymptotic direction.
\item[(ii)] $M^2_{\sing}\subset \mathbb R^3$ has a singularity worse than a cross-cap.
\item[(iii)] The curvature parabola $\Delta_p$ of $M^2_{\sing}\subset \mathbb R^3$ is degenerate.
\end{itemize}

The equivalence between (i) and (ii) can be found in \cite{mondthesis, BruceNogueira}, the equivalence between (ii) and (iii) is shown in \cite{MartinsBallesteros}. For $M^2_{\reg}\subset \mathbb R^5$ projected to $M^2_{\sing}\subset \mathbb R^4$ we have the same situation changing the cross-cap for the $I_1$-singularity (see \cite{Fuster/Ruas/Tari} and \cite{Benedini/Sinha/Ruas}). By Corollary \ref{nullasymp} adapted to these dimensions (the proof is the same) the direction $u\in T_q\tilde M$ is an asymptotic direction and in fact is the infinite asymptotic direction. This is why an image by $\eta$ of this direction is only defined in the case that $\Delta_p$ is degenerate.

For $M^3_{\reg}\subset \mathbb R^6$ projected to $M^3_{\sing}\subset \mathbb R^5$ we can prove the analogous result of the equivalence between (i) and (ii) as follows. In \cite{BenediniRuasSacramento} it is shown that, given a normal form $f(x,y,z)$ with the 2-jet $j^2f(0)$ of type $$(x,y,a_{20}x^2+a_{11}xy+a_{02}y^2+a_{21}z^2+a_{22}xz+a_{12}yz,b_{20}x^2+b_{11}xy+b_{02}y^2+$$$$b_{21}z^2+b_{22}xz+b_{12}yz,c_{20}x^2+c_{11}xy+c_{02}y^2+c_{21}z^2+c_{22}xz+c_{12}yz),$$ then it is $\mathscr A^2$-equivalent to one of the following orbits: $$(x,y,xz,yz,z^2),(x,y,z^2,xz,0),(x,y,xz,yz,0),(x,y,z^2,0,0),(x,y,xz,0,0),(x,y,0,0,0).$$ Furthermore, they show that it is in the best $\mathscr A^2$-orbit $(x,y,xz,yz,z^2)$ if and only if $\det(\alpha)\neq 0$ where $$\alpha=\left(
                                      \begin{array}{ccc}
                                        a_{21} & a_{22} & a_{12} \\
                                        b_{21} & b_{22} & b_{12} \\
                                        c_{21} & c_{22} & c_{12} \\
                                      \end{array}
                                    \right)
$$
For simplicity we take Monge forms and the direction of projection $u\in T_{p'}M^3_{\reg}$ to be $(0,0,1)$. Notice that in this setting $u$ is the null tangent direction in $T_q\tilde M$.

\begin{prop}
The direction $u=(0,0,1)\in T_pM^3_{\reg}$ is an asymptotic direction if and only if $j^2P_u(0)$ is not in the orbit $(x,y,xz,yz,z^2)$, where $P_u$ stands for the projection of $M^3_{\reg}$ in the direction $u$.
\end{prop}
\begin{proof}
Consider $M^3_{\reg}\subset \mathbb R^6$ given in Monge form $$f(x,y,z)=(x,y,z,f_1(x,y,x),f_2(x,y,x),f_3(x,y,x))$$ with 2-jet as above, 
then $P_v(x,y,z)=(x,y,f_1(x,y,x),f_2(x,y,x),f_3(x,y,x))$ is in the orbit $(x,y,xz,yz,z^2)$ if and only if $\det(\alpha)\neq 0$.

On the other hand, $u$ is asymptotic if there exists a non zero $\nu=(\nu_1,\nu_2,\nu_3)\in N_pM^3_{\reg}$ such that $u\in \ker \Hess h_{\nu}$, where $h_{\nu}(x,y,z)=\langle f(x,y,z),\nu\rangle=f_1\nu_1+f_2\nu_2+f_3\nu_3$. We have $$\Hess h_{\nu}\left(\begin{array}{c}
                                        0 \\
                                        0 \\
                                        1 \\
                                      \end{array}\right)=\left(\begin{array}{c}
                                        a_{22}\nu_1+b_{22}\nu_2+c_{22}\nu_3 \\
                                        a_{12}\nu_1+b_{12}\nu_2+c_{12}\nu_3 \\
                                        2(a_{21}\nu_1+b_{21}\nu_2+c_{21}\nu_3) \\
                                      \end{array}\right)=\left(\begin{array}{c}
                                        0 \\
                                        0 \\
                                        0 \\
                                      \end{array}\right)$$
if and only if $$\det\left(\begin{array}{ccc}
                                        a_{22} & b_{22} & c_{22} \\
                                        a_{12} & b_{12} & c_{12} \\
                                        a_{21} & b_{21} & c_{21} \\
                                      \end{array}\right)=\det(\alpha)=0.$$
\end{proof}

However, equivalence between (ii) and (iii) for 3-manifolds is not true in general (see \cite{BenediniRuasSacramento} for a partial result):

\begin{ex}
Consider $M^3_{\reg}\subset \mathbb R^6$ given by $(x,y,z,x^2+z^2,xy+xz,y^2)$. $u=(0,0,1)\in \ker \Hess h_{\nu}$ for the binormal direction $\nu=(0,0,1)$, so $u$ is an asymptotic direction. Projection along $u$ yields $(x,y,x^2+z^2,xy+xz,y^2)$ which is $\mathscr A^2$-equivalent to $(x,y,z^2,xz,0)$, which is not the best $\mathscr A^2$-orbit. However, the curvature locus is given by $$(\cos(\theta)^2+\frac{\cos(\phi)^2}{\sin(\phi)^2},\cos(\theta)\sin(\theta)+\cos(\theta)\frac{\cos(\phi)}{\sin(\phi)},\sin(\theta)^2)$$ which is not contained in a plane, i.e. it is not a degenerate curvature locus.
\end{ex}

The previous example shows that in some cases a non-degenerate curvature locus has an infinite asymptotic direction, so we must define the image and tangent space of $u_{\infty}$ for some $M^3_{\sing}\subset \mathbb R^5$ with non-degenerate curvature locus.

\begin{definition}
Let $\eta(\theta,\phi)$ denote the parametrisation of the curvature locus $\Delta_{cv}$ of $M^3_{\sing}$. For each topological type of the curvature locus we must define $\eta(u_{\infty})$:
\begin{itemize}
\item[(i)] If $\Delta_{cv}$ is a point $r$, then $\eta(u_{\infty})=r$ and $\frac{\partial\eta}{\partial \theta}(u_{\infty})=\frac{\partial\eta}{\partial \phi}(u_{\infty})=0$.
\item[(ii)] If $\Delta_{cv}$ is a line or a half line, then $\eta(u_{\infty})=\frac{\partial\eta}{\partial \theta}(u_{\infty})=\frac{\partial\eta}{\partial \phi}(u_{\infty})=\frac{\eta'(t)}{|\eta'(t)|}$ for any $t$ such that $\eta'(t)\neq 0$, where $t$ is the parameter of the line.
\item[(iii)] If $\Delta_{cv}$ is a planar region or a plane, then $\eta(u_{\infty})=\frac{\partial\eta}{\partial \theta}(u_{\infty})=\frac{\partial\eta}{\partial \theta}(v)$ for any $v$ such that $\eta(v)$ does not lie in the boundary of $\Delta_{cv}$ and $\frac{\partial\eta}{\partial \phi}(u_{\infty})=(\frac{\partial\eta}{\partial \theta}(v))^{\perp}$.
\item[(iv)] If $\Delta_{cv}$ is non-degenerate such that $u_{\infty}$ is an asymptotic direction of $M^3_{\reg}$, then $\eta(u_{\infty})=\frac{\partial\eta}{\partial \theta}(u_{\infty})=\frac{\partial\eta}{\partial \phi}(u_{\infty})=\lim_{\phi\rightarrow 0}\frac{\eta(\theta,\phi)}{|\eta(\theta,\phi)|}$.
\end{itemize}
\end{definition}

\begin{teo}\label{tangent}
Let $u\in C_q\cup \{u_{\infty}\}$, $u$ is asymptotic if and only if
\begin{itemize}
\item[(4)] The vector $\eta(u)$ is tangent to $\eta(C_q\cup\{u_{\infty}\})$ at $\eta(u)$ or $\Delta_{cv}=\eta(C_q)$ is singular at $u$.
\end{itemize}
\end{teo}
\begin{proof}
Suppose first that $u$ is not the null tangent direction. In this case we proceed first as in \cite{dreibelbis}. Let $C_q$ be parameterised by $(\theta,\phi)$ in cylindrical coordinates and let $u=u(\theta,\phi)$. Then $\frac{\partial\eta(u)}{\partial \theta}=II(u,u)_{\theta}=2II(u,u_{\theta})$ and $\frac{\partial\eta(u)}{\partial \phi}=II(u,u)_{\phi}=2II(u,u_{\phi})$. Since $\{u,u_{\theta},u_{\phi}\}$ are linearly independent, having the tangency or a singularity means that $\{II(u,u),II(u,u)_{\theta},II(u,u)_{\phi}\}=\{II(u,u),2II(u,u_{\theta}),2II(u,u_{\phi})\}$ is linearly dependent, and this happens if and only if there exists a unit tangent vector $w\in C_q$ such that $II(u,w)=0$. This is equivalent to the fact that there exists $w$ such that $II_{\nu}(u,w)=0$ for all $\nu\in N_pM$. Considering now $g=II_{\nu}(u,\cdot):T_q\tilde M\rightarrow N_pM$, since $w\in \ker g$, the image of $g$ is contained in a plane in $N_pM$ and so, what we have is equivalent to the fact that there exists $\nu\in N_pM$ such that $II_{\nu}(u,v)=0$ for all $v\in T_q\tilde M$, i.e. $u$ is asymptotic.

For $u_{\infty}$ the tangency occurs by construction of $\frac{\partial\eta}{\partial \theta}(u_{\infty})$ and $\frac{\partial\eta}{\partial \phi}(u_{\infty})$.
\end{proof}

\begin{teo}\label{asympnormalsection}
Let $M^2_{\sing}\subset \mathbb R^4$ be the normal section given by $\{Y=0\}$ of $M^3_{\sing}$, $p\in M^3_{\sing}$, and suppose that $Aff_{i_2^{-1}(p)}=E_{i_2^{-1}(p)}$, then $u\in T_{q}\tilde{M}$ is an asymptotic direction of $M^3_{\sing}$ if and only if $(i_{2_*})^{-1}(u)$ is an asymptotic direction of $M^2_{\sing}$.
\end{teo}
\begin{proof}
Lemma 4.10 in \cite{Benedini/Sinha/Ruas} is the equivalent result to Theorem \ref{tangent} for singular surfaces in $\mathbb R^4$. Taking a normal section is taking a hyperplane $U$ in $T_q\tilde M$. This induces an intersection of $\Delta_{cv}$ with the plane $II(U)$. A direction $u\in T_q\tilde M$ is asymptotic if $\eta(u)$ is tangent to $\eta(C_q\cup\{u_{\infty}\})$ or $\eta(C_q)$ is singular at $u$. The curvature parabola of $M^2_{\sing}\subset \mathbb R^4$ is given by $\Delta_{cv}\cap II(U)$ and here $\eta(u)$ is tangent to $\Delta_{cv}\cap II(U)$ if and only if $II(U)$ is a plane that passes through the origin, i.e. $Aff_{i_2^{-1}(p)}=E_{i_2^{-1}(p)}$. Hence, $\eta_{|_{U\cap C_q}}((i_{2_*})^{-1}(u))\in E_{i_2^{-1}(p)}$ and is also parallel to $\eta_{|_{U\cap C_q}}'((i_{2_*})^{-1}(u))$, that is, $(i_{2_*})^{-1}(u)$ is an asymptotic direction of the normal section.
\end{proof}

\begin{ex}
Consider the singular 3-manifold given by $(x,y,x^2-2yz,y^2-2xz,z^2-2xy)$. Its curvature locus is given by $$(2\cos(\theta)^2 - 4\sin(\theta) \frac{\cos(\phi)}{\sin(\phi)},
2 \sin(\theta)^2 - 4\cos(\theta) \frac{\cos(\phi)}{\sin(\phi)},
2 \frac{\cos(\phi)^2}{\sin(\phi)^2} - 4\cos(\theta) \sin(\theta)).$$ We have that $(0,1,-1)$ is an asymptotic direction associated to the binormal direction $(-1,1,1)$. Consider now the normal section given by $\{X=0\}$ and parameterised by $(y,-2yz,y^2,z^2)$. The curvature parabola is given by $(-4y,2,2y^2)$. Here $(-1,1,1)$ is a degenerate direction, but it is not binormal since it is not in $E_p$, and therefore, $(1,-1)$ is not an asymptotic direction of the singular surface.
\end{ex}

\begin{rem}
A similar result to Theorem \ref{asympnormalsection} for the regular case is not clear. The definition of asymptotic directions in $M^2_{\reg}\subset \mathbb R^{5}$ is slightly different from the rest of definitions. Namely, it depends on higher order singularities of the height function and therefore this is not second order geometry. The relation of these asymptotic directions with the asymptotic directions of $M^3_{\reg}\subset \mathbb R^{6}$ or of $M^2_{\sing}\subset \mathbb R^{4}$ is left for future work.
\end{rem}

\begin{rem}
Theorems \ref{equiv}, \ref{asympregsing} and \ref{tangent} can be generalised to other dimensions. The setting in this section works well for $M^n_{\reg}\subset \mathbb R^{2n}$ projected to $M^n_{\sing}\subset \mathbb R^{2n-1}$ and taking normal sections of the singular manifold to get $M^{n-1}_{\sing}\subset \mathbb R^{2(n-1)}$.
\end{rem}

\end{document}